\documentclass[a4paper,11pt]{article}
\author{Jelena~Jovanovi\'c\\ \emph{University of Belgrade, Faculty of Mathematics }}
\title{On strong Mal'cev conditions for congruence meet--semidistributivity \\ in a locally finite variety }
\usepackage{amsmath}
\usepackage{amssymb}
\usepackage{amsthm}
\usepackage{verbatim}
\date{}
\theoremstyle{plain}
\newtheorem{teorema}{Theorem}[section]
\newtheorem{definicija}[teorema]{Definition}
\newtheorem{tvrdjenje}[teorema]{Proposition}
\newtheorem{fakt}[teorema]{Fact}

\begin{document}
\maketitle
\begin{abstract}
In this paper we examine four--element and five--element digraphs for existence of certain polymorphisms that imply congruence meet--semidistributivity in a locally finite variety. The results presented here occurred  as an integral  part of my  research for optimal strong  Mal'cev conditions  describing  the property mentioned. Most of the programming needed to obtain these results was done in C programming language but for some results we also used the model builder Paradox. The source codes and Paradox inputs will be  presented here as well (appendices one and two of this paper contain the source codes; an example of Paradox input is given in section 6). 
\end{abstract}
\section{Introduction}

The various conditions that are equivalent to  congruence meet--semidistributivity in locally finite varieties of algebras have been examined in several papers and books so far -- see \cite{Czedli}, \cite{hm}, \cite{ksz}, \cite{lipp},\cite{rossw}, \cite{hhhm}.  Congruence meet-semidistributive varieties proved to be a very general yet very well behaved class of varieties. This condition is, for example, equivalent to congruence neutrality (\cite{ksz}, \cite{lipp}), or   to having no covers of types {\bf 1} or {\bf 2} in congruence lattices of finite algebras in the variety (holds for locally finite varieties, \cite{hm}). It also implies the truth of Park's Conjecture, \cite{park} (as proved by Ross  Willard in  \cite{rossw}), and also Quackenbush's Conjecture , \cite{quack} (which holds trivially in congruence distributive case due to J\'onsson's Lemma,   \cite{bjarni}, and is proved for the congruence meet-semidistributive case by Kearnes and Willard in  \cite{kearnesross}). Recently, the research in the Constraint Satisfaction Problem has shown that congruence meet-semidistributivity of the variety generated by the algebra of compatible operations is equivalent to the condition that the particular algorithm called 'localconsistency checking' would faithfully solve the Constraint Satisfaction Problem. This  property is called 'bounded width' and there is a lot of literature on the concept. This result due to L. Barto and M. Kozik (\cite{bartokozik}) is certainly among the strongest known partial results for the Dichotomy Conjecture, and probably the hardest to be proven so far.

Characterization of various semantical properties of all algebras in a variety and/or their congruence lattices  by equivalent syntactical conditions was started by A. I. Mal'cev in \cite{malcev}. Because of that, the properties which can be characterized in this way are called Mal'cev properties, and the corresponding syntactical  conditions - Mal'cev conditions. Particularly, if a property can be described by  a fixed number of term operations of fixed arities satisfying  a fixed number of linear equations (like the original Mal'cev condition for congruence permutability), we call it a {\em strong} Mal'cev property. On the other hand, a usual Mal'cev property (also known as {\em weak}) is equivalent to satisfying one strong Mal'cev condition (for some $n \in \omega$) from a given countable sequence of strong Mal'cev  conditions (for every $n \in \omega$) in which each member implies the next one (meaning that the properties are increasingly more general), like in the case of J\'onsson's condition for congruence distributivity, \cite{bjarni}. Obviously, strong Mal'cev conditions are preferable, if available, as then we can use the operations in a computer search, but there are properties which are Mal'cev properties but are proved not to be strong Mal'cev properties. The condition most commonly used for congruence meet-semidistributivity of a variety (not necessarily locally finite), until recently, was the one proved by Willard, \cite{rossw}, but the research in the Constraint Satisfaction Problem has recently uncovered that the congruence meet-semidistributivity of a locally finite variety is a strong Mal'cev property. The best, i. e. syntactically strongest we know of, is due to M. Kozik (\cite{hhhm}). We tried to see if it is also the best possible, see \cite{jj}, and we identified a single candidate system which implies congruence meet-semidistributivity and is syntactically stronger and with fewer operations and/or of smaller arity than the condition proved in \cite{hhhm}. We proved that either this system we found is indeed equivalent to congruence meet-semidistributivity of a locally finite variety, or the condition from \cite{hhhm} is the best possible (proved in \cite{jj}). However, which of these alternatives is true we have  not managed to ascertain yet. In this paper we examine algebras of polymorphisms of four--element and five--element digraphs searching for a counterexample for our system.

\section{Background}

In this paper an {\em algebra} denotes a structure ${\bf A}=(A,F^{\bf A})$, where $F$ is a signature, or language, consisting only of operation symbols of various arities, $A$ is a nonempty set, and for each symbol $f\in F$ of arity $k$ the corresponding element $f^{\bf A}\in F^{\bf A}$ is a mapping $f^{\bf A}:A^k\rightarrow A$. The set of {\em term operations} of ${\bf A}$ is the set of all operations obtained from $F^{\bf A}$ and projection operations via finitely many compositions. All algebras of the same signature which identically satisfy a set of equations are called a variety. An algebra ${\bf A}$ is locally finite if for any finite subset $X$ of $A$, the set of all results of term operations applied to elements of $X$ (that is a subalgebra generated by $X$)  is also finite. A variety is locally finite if every algebra in it is.

There is a natural connection between operations and relations on the same set. It says that a ($k$-ary) relation and an ($n$-ary) operation are compatible if for any $n$ vectors from the relation, the vector obtained by pointwise application of the operation is again in the relation. The classic results of universal algebra often connect the properties of the compatible equivalence relations, which form a lattice under inclusion called the congruence lattice, and other properties of algebras. In the paper (\cite{jj}) we mention the meet-semidistributivity of the congruence lattices of all algebras in a variety, which is the lattice implication $x\wedge z = y\wedge z \Rightarrow (x\vee y)\wedge z = x\wedge z$. An equivalent condition of the congruence meet-semidistributivity of a locally finite variety is omitting types of covers ${\bf 1}$ and ${\bf 2}$ in finite algebras of the variety (\cite{hm}). For any other definitions and basic results which are not found in this introductory part, the reader is referred to \cite{hhhhm} for basic universal algebra and \cite{hm} for tame congruence theory.

\begin{definicija}
Let $\mathbf{A}$ be a finite algebra and $\alpha$ a minimal congruence of $\mathbf{A}$ (i.e. $0_{\mathbf{A}} <  \alpha$ and if $\beta$ is a congruence of $\mathbf{A}$ with $0_{\mathbf{A}} <  \beta \le \alpha $ then $\beta = \alpha$.)
\begin{itemize}
\item An $\alpha$--minimal set of $\mathbf{A}$ is a subset $U$ of $\mathbf{A}$ that satisfies following two conditions:
\begin{itemize}
\item[-]$U = p(\mathbf{A})$ for some unary polynomial $p(x)$ of $\mathbf{A}$ that is not constant on at least one $\alpha$--class
\item[-] with respect to containment, $U$ is minimal having this property.
\end{itemize}
\item An $\alpha$--neighbourhood (or $\alpha$--trace) of $\mathbf{A}$ is a subset $N$ of $\mathbf{A}$ such that:
\begin{itemize}
\item[-] $N = U \cap (a /_{\alpha})$ for some $\alpha$--minimal set $U$ and $\alpha$--class $a /_{\alpha}$
\item[-] $\vert N \vert > 1$.
\end{itemize}
\end{itemize}
\end{definicija}
We can easily see that a given $\alpha$--minimal set $U$ must contain at least one, and possibly more, $\alpha$--neighbourhoods.The union of all $\alpha$--neighbourhoods in $U$ is called the body of $U$, and the remaining elements of $U$ form the tail of $U$. What is important here is that algebra $\mathbf{A}$ induces uniform structures on all its $\alpha$--neighbourhoods, meaning they (the structures induced) all belong to the same of five possible types. Let us now define an induced structure.
\begin{definicija}
Let $\mathbf{A}$ be an algebra and $U \subseteq \mathbf{A}$. The algebra induced by $\mathbf{A}$ on $U$ is the algebra with universe $U$ whose basic operations consist of the restriction to $U$ of all polynomials of $\mathbf{A}$  under which $U$ is closed. We denote this induced algebra by $\mathbf{A} \vert_{U}$.
\end{definicija}
\begin{teorema}
Let $\mathbf{A}$ be a finite algebra and $\alpha$ a minimal congruence of $\mathbf{A}$.
\begin{itemize}
\item If $U$ and $V$ are $\alpha$--minimal sets then $\mathbf{A} \vert_{U}$ and $\mathbf{A} \vert_{V}$ are isomorphic and in fact there is a polynomial $p(x)$ that maps $U$ bijectively onto $V$.
\item If $N$ and $M$ are $\alpha$--neighbourhoods then $\mathbf{A} \vert_{N}$ and $\mathbf{A} \vert_{M}$ are isomorphic via the restriction of some polynomial of $\mathbf{A}$.
\item If $N$ is $\alpha$--neighbourhood then $\mathbf{A} \vert_{N}$ is polynomially equivalent to one of:
\begin{enumerate}
\item A unary algebra whose basic operations are all permutations (unary type);
\item A one--dimensional vector space over some finite field (affine type);
\item A $2$--element boolean algebra (boolean type);
\item A $2$--element lattice (lattice type);
\item A $2$--element semilattice (semilattice type);
\end{enumerate}
\end{itemize}
\end{teorema}
\begin{proof}
The theorem in this form is given in \cite{nm}, and the proof can be found in \cite{hm}.
\end{proof}
The previous theorem allows us to assign a type to each minimal congruence $\alpha$ of an algebra according to the behaviour of the $\alpha$--neighbourhoods (for example, a minimal congruence whose $\alpha$--neighbourhoods are polynomially equivalent to a vector space is said to have affine type or type $2$).

Taking this idea one step further, given a pair of congruences $(\alpha, \beta)$ of $\mathbf{A}$ with $\beta$ covering $\alpha$ (i.e. $\alpha < \beta$ and there are no congruences of $\mathbf{A}$ strictly between the two), one can form the quotient algebra $\mathbf{A}/_{\alpha}$, and then consider the congruence $\beta /_{\alpha} = \{(a/_{\alpha}, b/_{\alpha}):(a,b) \in \beta \}$. Since $\beta$ covers $\alpha$ in the congruence lattice of $\mathbf{A}$,  $\beta /_{\alpha}$ is a minimal congruence of $\mathbf{A}/_{\alpha}$, so it can be assigned one of the five types. In this way we can assign to each covering pair of congruences of $\mathbf{A}$ a type (unary, affine, boolean, lattice, semilattice, or 1, 2, 3, 4, 5 respectively). Therefore, going through all covering pairs of congruences of this algebra we obtain a set of types, so--called typeset of $\mathbf{A}$, denoted by $typ \{\mathbf{A}\}$. Also, for $\mathcal{K}$  a class of algebras, the typeset of $\mathcal{K}$ is defined to be the union of all the typesets of its finite members, denoted by $typ \{\mathcal{K} \}$.

A finite algebra or a class of algebras is said to omit a certain type if that type does not appear in its typeset. For locally finite varieties omitting certain types can be characterized by Maltsev conditions, i.e. by the existence of certain terms that satisfy certain linear identities, and there are quite a few results on this so far. We shall present two of them concerning omitting types 1 and 2.
\begin{definicija}
An n--ary term $t$, for $n > 1$, is a near--unanimity term for an algebra $\mathbf{A}$ if the identities $t(x,x,\dots ,x,y) \approx t(x,x,\dots ,y,x) \approx  \dots  \approx t(x,y,\dots ,x,x) \approx t(y,x,\dots ,x,x) \approx x $ hold in $\mathbf{A}$.
\end{definicija}
\begin{definicija}
An n--ary term $t$, for $n > 1$, is a weak near--unanimity term for an algebra $\mathbf{A}$ if it is idempotent and the identities $t(x,x,\dots ,x,y) \approx t(x,x,\dots ,y,x) \approx  \dots  \approx t(x,y,\dots ,x,x) \approx t(y,x,\dots ,x,x)$ hold in $\mathbf{A}$.
\end{definicija}
\begin{teorema}

A locally finite variety $\mathcal{V}$ omits the unary and affine types (i.e. types 1 and 2) if and only if there is some $N > 0$ such that for all $k > N$, $\mathcal{V}$ has a weak near--unanimity term of arity $k$.\label{th}

\end{teorema}
\begin{proof} The proof can be found in \cite{hhm}.
\end{proof}
\begin{teorema}
A locally finite variety $\mathcal{V}$ omits the unary and affine types if and only if it has 3--ary and 4--ary weak near--unanimity terms, $v$ and $w$ respectively, that satisfy the identity $v(y,x,x) \approx w(y,x,x,x)$.
\end{teorema}
\begin{proof}
The proof can be found in \cite{hhhm}.
\end{proof}

Therefore, omitting types 1 and 2 for a locally finite variety ( or, equivalently, congruence  meet-semidistributivity for a locally finite variety ) can be described by linear identities on 3--ary and a 4--ary term, both idempotent. In the  paper \cite{jj} we examined whether the same can be done by two at most 3--ary idempotent terms. We came to this result:

\section{ A system possibly describing congruence meet--semidistributivity,\\ i.e.  omitting unary and affine types } \label{section two}

 \begin{tvrdjenje}
 \noindent If it is possible to describe congruence meet--semidistributivity (i.e. omitting types 1 and 2 ) in a locally finite variety by two ternary terms $p$ and $q$, it can only be done by this system:
\vspace{0.7 cm} 
 
 \begin{equation}
\left\{
\begin{array}{r}
p(x,x,y)\approx p(x,y,y)\\
p(x,y,x)\approx q(x,x,y) \approx q(x,y,x) \approx q(y,x,x)
\end{array}\right. \label{system2}
\end{equation}
\end{tvrdjenje}

\vspace{0.7 cm}

\noindent It is easy too see that this system implies the property mentioned (the proof can be found in \cite{jj}), but it  remained unresolved whether it actually describes the property. In attempt to find a counterexample , i.e. a finite algebra that generates a variety satisfying congruence meet--semidistributivity but not satisfying the system above, we endeavor  to examine small digraphs, i.e. corresponding algebras of polymorphisms.
\begin{definicija}
A digraph is a pair $G=(V,E)$, where $G$ is a finite set of vertices and $E\subseteq V \times V$ is a set of edges. 
\end{definicija}
\begin{definicija}
An n--ary polymorphism of a digraph $G=(V,E)$ is a mapping $f:V^n \rightarrow V$ which preserves edges, that is for any $(a_{1},b_{1}), (a_{2},b_{2}), \dots ,(a_{n},b_{n}) \in E  $ the pair $(f(a_{1},\dots a_{n}), f(b_{1},\dots b_{n})) $ is also in $E$.
\end{definicija}
\begin{definicija}
An algebra of polymorphisms for a given  digraph $G=(V,E)$ is an algebra with the universe $V$ whose basic operations are polymorphisms of this digraph. 
\end{definicija}

\noindent Further research eplained here leans on results obtained by  Libor Barto and David Stanovsky as presented in \cite{bs}. As said before, congruence meet--semidistibutivity of a locally finite variety is equivalent to omitting covers of types {\bf 1} or {\bf 2} in the congruence lattices of finite algebras in the variety (\cite{hm}) , and is also referred to as 'bounded width' due to the fact that 'localconsistency checking' algorithm can faithfully solve the Constraint Satisfaction Problem in this case. So we say that a finite algebra is of 'bounded width' if and only if it generates a congruence meet--semidistributive variety. 

\section{ Two--element and three--element digraphs } \label{section three}

 There are $10$ non--isomorphic digraphs on  two vertices and each of them has an 3--ary  near--unanimity polymorphism, also called 'nu3' or a majority polymorphism (\cite{bs}). A finite algebra having a majority term--operation generates a congruence meet--semidistributive variety, but also satisfies  system \ref{system2} ( this is proved in \cite{jj}, example two). Therefore algebras of polymorphisms of two--element digraphs are of no further  interest to us.
\vspace{0.7 cm} 

\noindent As for three--element digraphs, we need to explain the case a bit more thoroughly,  again referring to results in \cite{bs}. We shall first define a two--semilattice operation:
\begin{definicija}
A 2--semilattice (2--sml) operation is an idempotent operation satisfying $f(x,y)\approx f(y,x)$ and $f(f(x,y),x)\approx f(x,y)$.
\end{definicija}
\begin{fakt}A finite algebra having a 2--semilattice term--operation generates a congruence meet--semidistributive variety, but also satisfies  system \ref{system2}.
\end{fakt}
\noindent  This is proved in \cite{jj}, example one.

\vspace{0.7 cm} 
 
\noindent Results given in \cite{bs}  include figures of posets comparing the strength of polymorphisms. From these we can conclude the following about three --element digraphs, i.e. their algebras of polymorphisms: if they have a  'bounded width' polymorphisam  they either have a  nu3 (majority)   or a 2--semilattice polymorphisam (or both).  Here  'bounded width polymorphism' stands for existence of both a 3--wnu and a 4--wnu polymorphism
 sharing the same binary operation as said in  definition 2.5 and theorem 2.7 above.  Therefore for all corresponding algebras of polymorphisms  holds the following: if they generate a congruence meet--semidistributive variety (i.e. if they have a  'bounded width polymorphism')  they  also have either a nu3 or a 2--sml polymorphisms (or both), thereby satisfying system \ref{system2}. This means they are of no further interest in our research.
 
 \section{ Four--element digraphs } \label{section four}
 
   If we take a look at the strength of polymorphisms of these digraphs    (figure $10$ in \cite{bs}) we can conclude the following: nu3 and 2--sml polymorphisms both imply 'bounded width' polymorphism (are stronger). Also the existence of a 2--sml polymorphism is equivalent to the existence of a weak near--unanimity polymorphism of arity 2, or wnu2 (in the sense that a digraph has the first one if and only if it has the second one). From the figure $11$ in the same paper it can be seen that there are $29$ digraphs on four vertices for which a 'bounded width' polymorphism is minimal and there are no digraphs of this size having wnu3 or wnu4 polymorphism as minimal. Since a 'bounded width polymorphism' is stronger than these two, we came to the following conclusion: if we exclude digraphs having a nu3 (majority) polymorphism and digraphs having a wnu2 polymorphism, remaining digraphs have a 'bounded width' polymorphism if and only if they have a wnu3 polymorphism (and there should be only $29$ of them satisfying this condition). We managed to isolate these $29$ digraphs and then tested them for existence of polymorphisms $p$ and $q$ as in the system (\ref{system2}). In the result we found out they all satisfy this system, so there is  no counterexample of size four.

\subsection{The procedure}
 
\noindent C--source code examining  four--element digraphs does the following (it is provided in Appendix one of this paper):
\begin{enumerate}
\item generates all digraphs of size 4
\item detects non--isomorphic ones 
\item sets the matrix of subalgebras for non--isomorphic digraphs 
\item then identifies ones having  a nu3 polymorphism (a majority polymorphism) and excludes them from further examination ( they are of bounded width but satisfy the system (\ref{system2}), so are of no interest here) 
\item among the remaining digraphs it identifies ones having a wnu2--polymorphism (i.e. a binary idempotent commutative operation) and excludes these too for the same reason as above 
\item the remaining digraphs are then tested on the existence of a wnu3 term, which gave us $29$ digraphs having a 'bounded width' polymorphism as a minimal one
\item we then test these digraphs for existence of polymorphisms from the system (\ref{system2})
\item in the final result all of these digraphs satisfy the system mentioned, so we have no counterexample among four--element digraphs
\end{enumerate}
\vspace{2 cm}
\noindent We shall provide here a bit more information on each item in the previous list (quite a detailed explanation is given with the source code in appendices, as well as a number of comments within this  code).

\vspace{0.5 cm}

\begin{enumerate}
\item We represented a 4--element digraph by a 16--element array (of character type), members being '1' for an edge  and '0' for no edge. The  presumed order of edges in this array would be  (0,0), (0,1), \dots ,(3,2), (3,3). Generating all digraphs of size four was done in the following way: we looked at each digraph as it was a binary number having exactly the same digits, so we could obtain the next digraph just by adding a binary $1$ to the current one. We started, of course, from the array containing $16$ zeroes (that is $16$ characters '0'). This way we generated all $65536$ digraphs with four vertices.
\item When detecting non--isomorphic digraphs we used the following procedure: we take the first digraph (from the array of all digraphs) having '0' on its isomorphism flag (not a copy), and then examine for isomorphism  all digraphs with greater indexes and the same number of edges. Each time we find an ismorphic copy, we set its flag to '1' (is a copy). When all the copies are identified we proceed through the  array  to the next digraph having '0' on this flag and do the same. The procedure used is basically a  'brute force' one (improved by the fact that we only examine digraphs with the same number of edges for isomorphism), but it works fast enough since the number of 4--element digraphs is not too big.  This way we detect all $3044$ non--isomorphic digraphs.  
\item A four element digraph can have maximum ten subalgebras, so we formed a ten--column matrix of character type to keep information on subalgebras for each digraph ( the number of rows is actually 65536, which is the number of all 4--element digraphs, but we only set values in rows corresponding to non--isomorphic digraphs).   The presumed order of subalgebras is  $\{0,1\}$, $\{0,2\}$, $\{0,3\}$, $\{1,2\}$ \dots,$\{1,2,3\}$, and the matrix may  contain only  '0' for 'not a subalgebra'   and '1' for 'is a subalgebra'.  Subalgebras are set by the function doing four checks,  which means we did not check all possible conditions for subalgebras, but this worked just fine. Our function  does the following checks for  each one of non--isomorphic digraphs and for each of the sets above (say, for example, $\{0,2\}$):
\begin{itemize}
\item  we check  whether there is a node in this digraph such that $0$ and $2$ are the only nodes with edges leading to this node -- if this is a case then $\{0,2\}$ is a subalgebra of the digraph given, so we set the corresponding element of the matrix to '1' and skip the following checks for the same subset;  
\item  if the first check did not give us the subalgebra result we do the next one --  whether there is a node in the digraph  having only two edges leading from it and exactly to $0$ and $2$ -- if so $\{0,2\}$ is a subalgebra so '1' is being  set on the proper matrix element, and the remaining two checks skipped; 
\item the third condition -- whether $0$ and $2$ are the only two nodes having paths of length one leading to them -- if so we have a subalgebra, \dots, we skip the last check;  
\item the last condition --  whether $0$ and $2$ are the only two nodes having paths of length one leading from them -- if so we have a subalgebra,  \dots, if not we are finished with the set $\{0,2\}$ and   the corresponding element of the matrix has not changed (there remains '0', not a subalgebra).
\end{itemize}

\item  Examining digraphs for a majority polymorphism was done by a pair of  recursive functions doing backtracking on a $24$--element array. This array is of structure type, each of its elements containing a three element string of arguments for a majority polymorphism (the order is "012", "013", \dots, "321") and a character type variable which should contain a value assigned to the polymorphism function in these arguments. Initially all the values are set to 'a', which is, of course, not an allowed value for a polymorphism.  The functions operating on this array does the following: the first function, named  $f_m$, sets the first value (that would be '0') on the element of the array being set, then calls the 'check' function to test whether the value assigned is compatible with the relation represented by the digraph  and also with previously assigned values of this polymorphism (that is array). If we have compatibility the value is allowed, so $f_m$ function would just move to the next element  of the array being set and call itself again. In the opposite case, that is no--compatibility, $f_m$ function attempts to assign the value '1' and then calls for a check, if it is allowed it goes to the next element of the array\dots. If none of the values '0', '1', '2', '3' can be assigned to the current element of the array $f_m$ function calls the other recursive function, named 'backwards', to reset previously defined elements of the array (that is all from the beginning to the current one, not including the current one) if possible. When these values are reset $f_m$ will start again on the current element, beginning with '0'. If it was impossible to reset the preceding values this function will finish, leaving 'a' as a value on the current element of the array and also on all the following elements   , which would subsequently be detected as 'no majority polymorphism'. So basically there are two recursive functions -- one going forward, if possible, the other going backwards, that is resetting once set values in the array when detected that we can not move forward any more. When these functions are finished we have one of two possible outcomes:  if all  $24$ values are set it means we have a mojority polymorphism on the digraph being examined, and if there is the  value 'a' on the last element of the array it means there is no majority polymorphism (it is sufficient to test just the last value in the array once the functions finish).  This worked just fine and fast enough in a sequential mode and we obtained the result -- there are $1690$ digraphs having this polymorphism.
\item Examining digraphs for a wnu2 polymorphism was done in the following way:  first we create a matrix containing all idempotent commutative binary operations on four elements. A binary idempotent commutative operation needs only be defined on six pairs of arguments, i.e. $(0,1), (0,2), (0,3),(1,2), (1,3), (2,3)$, so each operation was presented by a six--element row of this matrix containing its  values in these pairs of arguments. For each digraph we go through the matrix and check compatibility -- this is done by the function $f_ {pair}$. We shall give a short example of how checking compatibility works: say we are examining an operation (presented by a row of the matrix mentioned) having $0$ as a value in $(1,2)$, that is $f(1,2)=0$. For all nodes $x$ and $y$ of this digraph such that there is an edge from $x$ to $1$ and from $y$ to $2$ there should be an edge from $f(x,y)$ to $0$. The other way round should also hold, that is for  all nodes $x$ and $y$  such that there is an edge from $1$ to $x$ and from $2$ to $y$ there should be an edge from $0$ to $f(x,y)$.  
Once we find a compatible operation we proceed to the next digraph. 
\item The existence of wnu3 polymorphism, however, could not be examined in the way  we did it for a majority polymorphism -- namely backtracking explained above was in this case  to be done on a 36--element array because a wnu3 polymorphism needs to be defined on the same $24$ arguments like a majority polymorphism, that is "012", "013", \dots, "321", but also on these: "001",  "002",\dots, "330", "331", "332".  The method used for a majority term  did not work fast enough (sequential mode), because  the backtracking array was too long, thereby generating too many recursive calls of the functions mentioned. We had to shorten 'the backtracking part' of this array, that is to assign some values to some of the arguments  for this polymorphism and then attempt backtracking on the rest. The way we did this was the following: 
\begin{itemize}
\item we identified 2--element subalgebras of the digraphs (these were the remaining digraphs, that is the ones not having a majority nor a wnu2 polymorphism)
\item The digraphs can have anywhere in between $0$ and $6$ two--element subalgebras, so we divided them into categories  according to this number. We shall explain the further procedure on two of these categories -- for digraphs having all six two--element subalgebras, and for those having, say,  five of these.   
\item For digraphs having six two--element subalgebras we looked at these arguments for a wnu3 polymorphism: "001", "002", "003", "110", '112", "113',\dots ,"330", "331", "332". In each of these twelve arguments the polymorphism can only have one of the two possible values : 0 or 1 in "001", 0 or 2 in "002" etc. We then generated a matrix  containing all possible values for these twelve arguments ($12$ columns and $2^{12}$ rows). These twelve strings of arguments are in the beginning of the backtracking array, so we assigned row by row of this matrix to be the values of the wnu3 polymorphism for these specific arguments and attempted backtracking from the thirteenth member of the array to the end for each of this rows. This enhanced the procedure greatly, since the recursive functions doing backtracking  were executed on a 24--member array. These functions had to be altered a bit in comparison to functions searching for a majority polymorphism -- namely they only set the array from some point on, never changing the beginning of it, but when setting a particular value they check compatibility with all previous values including the ones from the matrix. 

\item In the case of five two--element subalgebras the things get a bit more complicated: there are six possible choices of five subalgebras and for each of these choices we have to create a new matrix of values and to alter the beginning of the backtracking array. For example if subalgebras are $\{0,1\}$, $\{0,2\}$, $\{0,3\}$, $\{1,2\}$, $\{1,3\}$ the matrix would have five columns and $2^{5}$ rows,   having only $0$ and  $1$ in the first column, $0$ and  $2$  in the second and so on. Moreover we have to put ten corresponding elements, that is "001", "002", "003", "110", "112", "113", "220", "221", "330", "331",  in the beginning of the backtracking array ( putting elements  "223" and "332" on positions  eleven and twelve). Then we do the same as in the previous case, that is we assign  a row of the matrix to be the values on first ten members and attempt backtracking from the  eleventh member on. 
We have to create a new matrix of values and  to permute some elements of the backtracking array for each choice of five subalgebras, which means this part of the code needs to be executed within six iterations. In the functions doing backtracking we only need to change a starting point from the index $13$ to the index $11$ in the backtracking array.  This is explained thoroughly in appendices of this paper. The backtracking part of the array is in this case  $26$ members long and it works fast enough.  
\item In all the remaining cases, that is four, three, two or one two--element subalgebra we did exactly the same as explained above. In each of these cases we had to create a new matrix of possible values for each choice of subalgebras, and also to alter the beginning of the backtracking array so that the corresponding elements would come first. The functions searching for a wnu3 term would also start  from different positions in this array. This means we had to iterate this part of the code quite a number of times to obtain the final results.   
\item The procedure described worked very well-- luckily enough all the digraphs having a wnu3 polymorphism were among ones  actually having two--element subalgebras, so we managed to find all $29$ of them. What strikes as interesting is  that  assigning values to just two elements of the 36--element array and doing backtracking on the rest of it, on the length  $34$ that is, was fast enough (this is the case with a single 2--element subalgebra). As mentioned before,  backtracking on the length $36$ was impossible (that is did not finish even after several hours for particular digraphs).
\item This way we isolated  $29$ digraphs not having a majority nor a wnu2 polymorphism, but having a wnu3 polymorphism. As already said these are the ones having a 'bounded width' polymorphism as a minimal one, therefore were candidates for a counterexample.

\end{itemize}
\item In this  step we tested these digraphs for existence of polymorphisms $p$ and $q$ from system \ref{system2}. It was done in the following way:
\begin{itemize}
\item first we find a wnu3 term for the digraph being examined; this is to be the $q$ term from the system \ref{system2}.
\item then we assign $12$ values of this term  to the elements at the beginning of the backtracking array for the $p$ term (these would be the values of the wnu3 term for the following arguments: "010", "020", "030", "101", "121", "131", "202", "212", "232", "303", "313" and "323". In these arguments $p$ and $q$ have the same values according to the system \ref{system2}). 
\item next we do backtracking on the remaining of the array (that is we try to assign remaining $36$ values for $p$). The functions doing backtracking and all the checks needed are exactly the same as described before (in the cases of majority term and wnu3 term), so we shall not comment them again. They can be found within the source code in appendices of this paper.
\item in the result all $29$ digraphs were found to satisfy the system \ref{system2}, that is to have terms $p$ and $q$.

\end{itemize}

Interestingly enough, this backtracking on $36$ elements when searching for the term $p$ worked just fine, as opposed to the attempt on the same length when examining digraphs for a wnu3 polymorphism. Also, for each of these $29$ digraphs the term $p$ was found for the very first wnu3 term being examined (when we take the least wnu3 as the term $q$  for a particular digraph, the least by its values that is, we can find the term $p$ paired with this one so as to satisfy the system \ref{system2}).
\item as already said we established there was no counterexample among $4$--element digraphs.    
 
\end {enumerate}
\subsection{Complexity}

\noindent Now we shall estimate time complexity for each of the steps (i.e sections of code) mentioned above. 
\begin{enumerate}
\item The time needed for generating all digraphs of size four (or of any given size) is a linear function of the number of digraphs of given size. If we denote a number of  digraphs with $N$ generating would take $C*N$ time units, $C$ being some constant value. However, the number of digraphs is actually $N = 2^{2^{n}}$, $n$ being a number of vertices, so this step takes $C*2^{2^{n}}$ time units.
\item When searching for non--isomorphic digraphs we used a method close to the 'brute force' one as already said, and by 'brute force' we mean comparing a digraph for an isomorphism with all digraphs that follow this one in the array containing all digraphs. We improved this a bit by comparing only digraphs having the same number of edges, but basically the complexity of this method remained the same -- if $N$ is a number of all digraphs of given size this step takes $C*N^{2}$ time units, $C$ being some constant value. We can also express this as a function of number of vertices, so it would be $C*(2^{2^{n}})^{2}$    . This worked fine for 4--element digraphs, but for 5--element ones we used a better (and more clever) method as explained in the section \ref{section five}.
\item The function checking subalgebras takes constant amount of time for each digraph, which gives us $C*N$ time units, where $C$ is some constant value and $N$ is the number of non--isomorphic digraphs.
\item Searching for a majority polymorphism was done by a pair of recursive functions, as already explained above. The first one,  called $f_m$ in the original code,  sets the value of the current element of the backtracking array (that is a value for this polymorphism on a given $3-tuple$ of arguments), if possible, and then moves  forward in this array. The second one, called 'backwards', resets  all previously set values in this array to new ones when setting the current value is not possible. Let us look at the time complexity here, analyzing only the worst possible case of course: say we are setting a value in the element of the array having index $i$. If it is impossible to set it considering previously set values the function 'backwards'  can be called by the function $f_m$ as many as $4^{i}$ times before $f_m$ proceeds further (because this is the number of ways to set values on the preceding $i$ members of this array). Each call of the function 'backwards' can generate $4^{i-1}$ new calls of this function, and each of these calls can give us $4^{i-2}$ new calls... So, basically, when setting the element of the array with the index $i$ there can occur as many as $4^{i}*4^{i-1}*4^{i-2}*\dots *4 = 4^{i*(i+1)/2}$ calls of the function 'backwards'. So, if we suppose this is what happens when setting a value on each element of the array, we have $\Sigma_{i=1}^{i=M} 4^{i*(i+1)/2}$ calls of this function, $M$ being the length of the backtracking array, that is the number of values to be set for a certain polymorphism. We can approximate the previous sum, from above, to $\mathbf{C*4^{M*(M+1)/2}}$, $C$ being some constant value (since $\Sigma_{i=0}^{i=\infty} 4^{i} = 4/3 $). So the time complexity of this procedure would be $\mathbf{C*4^{M*(M+1)/2}}$ time units ($C$ a constant value, $M$ the number of values to be set for a certain polymorphism) , which is extremely high. There are a number of checks and loops accompanying this recursive calls, but they do not contribute significantly to the complexity obtained so we did not mention them for the sake of simplicity. The  backtracking described  worked for $M  = 24$, a majority  polymorphism needs $24$ values defined, but could not work for $M = 36$ when searching for a wnu3 polymorphism.  The breaking point was the length  $35$, or maybe even $36$ since it worked on the length  $34$, as we have mentioned above. 

\noindent At this point we could alter this complexity expression given above as to make it a function of the number of vertices -- if a digraph had $n$ vertices we would have  to define $n*(n-1)*(n-2)$ values for a majority polymorphism, so the backtracking array would be that long, and this gives us the following complexity: $\mathbf{C*n^{n*(n-1)*(n-2)*(n*(n-1)*(n-2)+1)/2}}$. 
\item When examining digraphs for a wnu2 polymorphism first we created a matrix of all idempotent commutative binary operations   . Its dimensions are $n^{n*(n-1)/2} \times n*(n-1)/2$, where $n$ is the number of vertices (in this case $4$). For each  digraph we went through this matrix checking compatibility, the check requiring a constant amount of time. Therefore time complexity of this section of code would be $C*n^{n*(n-1)/2}*N$, where $C$ is some constant value, $n$ is the number of vertices and $N$ is the number of digraphs being examined (non--isomorphic digraphs not having a majority polymorphism).
\item The time complexity of the section of the code searching for a wnu3 polymorphism is the same as when searching for a majority polymorphism for we used the same functions to do the backtracking needed, $\mathbf{C*4^{M*(M+1)/2}}$, where $C$ is some constant value and $M$ is the number of values to be set for a certain polymorphism (wnu3). We did shorten the backtracking part of the array from the length $ M = 36$ (impossible) to the length at most $34$ (possible) in the way described above, but generating the matrices needed and assigning values to the members at the beginning of the backtracking array do not effect overall complexity of this method -- it is very high due to the number of recursive calls. Iterations mentioned above and needed to obtain results  do not effect complexity either, they can just alter this constant value $C$ in the expression.
\item The time complexity of the section of the code searching for polymorphism $p$ and $q$  is even greater then when searching for a majority polymorphism or a wnu3 polymorphism since the idea is to find a wnu3 polymorphism first, then to search for a $p$ term for this particular  wnu3, and in case  it doesn't exist we search for another  wnu3 polymorphism and then for a corresponding $p$ term  and so on. This basically means that the complexity would be (the upper bound) $\mathbf{C*4^{M*(M+1)/2}*4^{M_{1}*(M_{1}+1)/2}}$, where $C$ is some constant value and $M$, $M_{1}$ are  the numbers of values to be set for these polymorphisms (wnu3 and $p$ respectively). The fact that in each case the term $p$ was found for the very first wnu3 term being examined made things much easier for us, of course, but it does not effect the complexity -- in the worst case it could be as given above.        
\end{enumerate}

 \section{ Five--element digraphs } \label{section five}
 
 For five-- element digraphs holds exactly the same as for four--element ones regarding the strength of polymorphisms (figure $10$ in \cite{bs}): if we exclude ones having a majority polymorphism and a wnu2 polymorphism  the remaining digraphs are of 'bouded width' if and only if they have a wnu3 polymorphism. There should be $3475$ of these digraphs, and exactly they would be the candidates for a counterexample.It is here that we encountered a problem -- namely we obtained results on a wnu2 polymorphism, but the source codes doing backtracking for 5--element digraphs were not fast enough, though ran in parallel mode, so we had to use Paradox model builder to examine these digraphs for majority, wnu3 and $p$ and $q$ polymorphisms.

What we did by C--source code was generating all 5--element digraphs, finding all non--isomorphic ones (this was done in sequential mode) and then examining these for an idempotent binary commutative operation (a wnu2 polymorphism, done in parallel mode).  The methods used are very similar to these when processing four--element digraphs, namely we set a matrix containing  subalgebras for each digraph, and yet another one containing all idempotent commutative binary operations. In this case each operation was represented by a 10--element row of this matrix. For each digraph we went through this matrix checking compatibility. Because of both the number of non--isomorphic digraphs and the number of binary operations we executed this code on $13$ processors in a master--slave hierarchy. Each one of the slaves was given its own copy of both matrices mentioned, while the master distributed digraphs one by one to each slave.

Both codes are presented and explained in details in Appendix two.

\vspace{0.7 cm}

\emph{The results we obtained on 5--element digraphs  slightly differed from the ones given in \cite{bs}  in the sense that we got $132509$ non--isomorphic digraphs having a wnu2 polymorphism, which is less by one than the number given in \cite{bs}. When examining the file the authors provided as a result of testing all  non--isomorphic 5--element digraphs for 18 different polymorphisms  we found this was no more than a counting mistake-- namely there were some redundant data at the end of this file. This however caused  all the entries in the figure $7$ in \cite{bs} to be greater by one then then actually are.}

\emph{The erratum is now on the website presenting   results from \cite{bs}.} 

\subsection{Complexity}

\noindent As in the previous section we shall estimate complexity for each part of the procedure:
\begin{itemize}
\item The time needed to generate all digraphs of size five is a linear function of the number of digraphs, so $C*N$, $N$ being the number of digraphs, or $C*2^{2^n}$, $n$ being the number of vertices (here five). 
\item Detecting non--isomorphic digraphs in the way explained above takes $C*N$ time units, $N$ being the number of all digraphs, because generating isomorphic copies for each digraph and then calculating their positions and setting corresponding flags to '1' takes a constant amount of time. This way we reduced the complexity for an order of magnitude, namely from $C*N^2$ (brute force) to $C*N$ ('the sieve of Eratosthen' method).
\item As in the case of 4--element digraphs, setting the matrix of subalgebras takes $C*N$ time units, $N$ being the number of non--isomorphic digraphs with five vertices. 
\item Setting the matrix of all idempotent commutative binary operations and examining non--isomorphic digraphs for this kind of operation (a wnu2 polymorphism) takes $C*n^{n*(n-1)/2}*N$ time units, where $n$ is the number of vertices, here five, and $N$ is the number of non--isomorphic digraphs of this size. The complexity expression is already explained in the 4--element case. Let us only notice that executing this section of code in parallel mode does not change its complexity but only the constant factor $C$, that is the actual amount of time needed.   
\end{itemize}

\subsection{Paradox}

\noindent As mentioned above we used Paradox model builder to examine remaining digraphs for majority, wnu3 and $p$ and $q$ polymorphisms (as denoted in the system \ref{system2}). Paradox was also executed in parallel mode under Linux shell. All input files were generated by C-programming language as text files (an input file depends on a digraph and a polymorphism). We shall present here an example of input file (this was used when searching for $p$ and $q$, $q$ being denoted by $g$):
\begin{verbatim}
cnf(mt, axiom, p(X,X,X)=X).
cnf(mt, axiom, p(X,X,Y)=p(X,Y,Y)).
cnf(pr, axiom, ~gr(X0,X1) | ~gr(X2,X3) |  ~gr(X4,X5) | gr(p(X0,X2,X4),p(X1,X3,X5))).
cnf(wnu, axiom, g(X,X,X)=X).
cnf(wnu, axiom, g(X,X,Y)=g(X,Y,X)).
cnf(wnu, axiom, g(X,X,Y)=g(Y,X,X)).
cnf(pr, axiom, ~gr(X0,X1) | ~gr(X2,X3) |  ~gr(X4,X5) | gr(g(X0,X2,X4),g(X1,X3,X5))).
cnf(mt, axiom, p(X,Y,X)=g(Y,X,X)).
cnf(graph, axiom, ~gr(n0,n0)).
cnf(graph, axiom, ~gr(n0,n1)).
cnf(graph, axiom, ~gr(n0,n2)).
cnf(graph, axiom, gr(n0,n3)).
cnf(graph, axiom, gr(n0,n4)).
cnf(graph, axiom, ~gr(n1,n0)).
cnf(graph, axiom, ~gr(n1,n1)).
cnf(graph, axiom, ~gr(n1,n2)).
cnf(graph, axiom, ~gr(n1,n3)).
cnf(graph, axiom, ~gr(n1,n4)).
cnf(graph, axiom, ~gr(n2,n0)).
cnf(graph, axiom, ~gr(n2,n1)).
cnf(graph, axiom, ~gr(n2,n2)).
cnf(graph, axiom, ~gr(n2,n3)).
cnf(graph, axiom, gr(n2,n4)).
cnf(graph, axiom, gr(n3,n0)).
cnf(graph, axiom, ~gr(n3,n1)).
cnf(graph, axiom, ~gr(n3,n2)).
cnf(graph, axiom, ~gr(n3,n3)).
cnf(graph, axiom, gr(n3,n4)).
cnf(graph, axiom, ~gr(n4,n0)).
cnf(graph, axiom, gr(n4,n1)).
cnf(graph, axiom, ~gr(n4,n2)).
cnf(graph, axiom, ~gr(n4,n3)).
cnf(graph, axiom, ~gr(n4,n4)).
cnf(elems, axiom, n0!=n1).
cnf(elems, axiom, n0!=n2).
cnf(elems, axiom, n0!=n3).
cnf(elems, axiom, n0!=n4).
cnf(elems, axiom, n1!=n2).
cnf(elems, axiom, n1!=n3).
cnf(elems, axiom, n1!=n4).
cnf(elems, axiom, n2!=n3).
cnf(elems, axiom, n2!=n4).
cnf(elems, axiom, n3!=n4).
cnf(elems, axiom, (X=n0 | X=n1 | X=n2 | X=n3 | X=n4)).
\end{verbatim}

\noindent  We managed to isolate $3475$ $5$--element digraphs having a wnu3 polymorphism as a minimal one, and they proved to have $p$ and $q$ terms, that is to satisfy the system \ref{system2}. This means we established there was no counterexample among $5$--element digraphs, i.e. among corresponding algebras of polymorphisms.

\section{the conjecture}

\noindent Since we found no counterexample for the system \ref{system2} up to size $5$, at this point we can state the conjecture:  
\vspace{0.7 cm}

\noindent \emph{A locally finite variety is congruence meet--semidistributive if and only if it satisfies the system  \ref{system2}}:  
 \begin{equation*}
\left\{
\begin{array}{r}
p(x,x,y)\approx p(x,y,y)\\
p(x,y,x)\approx q(x,x,y) \approx q(x,y,x) \approx q(y,x,x)
\end{array}\right. 
\end{equation*}

\vspace{2 cm}

\section{Appendix one -- Source code for 4--element digraphs, sequential mode}

What does it do:
\begin{itemize}
\item first we generate all 4-element digraphs
\item then we find all the non--isomorphic ones
\item in the next step we set the matrix of subalgebras for all  non--isomorphic digraphs -- knowing subalgebras of a digraph helps a great deal when searching for polymorphisms    
\item we identify digraphs having a majority term (we made a function doing backtracking on a 24--element array) and then  exclude these digraphs  from further examination
\item among the digraphs left we identify those having an idempotent   binary commutative operation (a wnu2 term) and exclude them too -- this is done by generating the matrix containing all binary idempotent commutative operations in which an operation is presented by a 6--element row, and then just  checking compatibility through the matrix for each digraph
\item the rest of digraphs are then tested for a wnu3 term, and the ones having it (exactly $29$ digraphs) are printed on the screen (this however could not be done simply bu doing backtracking  on a 36--element array as expected, but we had to shorten the backtracking part of the array by repeatedly assigning allowed values to several of its members at the beginning, according to the subalgebras of the digraph in question, and then attempting  the backtracking from that point on. This method works just fine -- it appears that fixing just four values of this array and doing backtracking on 32 members left creates no problem of whatsoever) 
\item  in the last step we tested the $29$ digraphs having a wnu3 term for terms $p$ and $q$ as explained above. They all appeared to have these terms, so in the result we established there was no counterexample of size four.   
\end{itemize}
\vspace{1 cm}

\noindent This code can be executed exactly as given, although quite a few changes of parameters, that is iterations, are needed to obtain all the results. This is, however, explained by  code--comments all along.

\vspace{2 cm}

\noindent \bf{The source code is the following:}

\vspace{2 cm}

\begin{verbatim} 
// fourvertices.cpp : Defines the entry point for the console application.
//

#include "stdafx.h"
#include <stdio.h>
#define NUMBER 65536  /* the number of all digraphs with four vertices*/
#define DIM 4096

/* structures*/

struct graph { char description[17] ; char isomorph; char majority; char wnu2; 
char wnu3;}; 
struct graph arr[NUMBER];
/*to each digraph we associate an array of lenth 10 for 10 subalgbras, 
the order is 01,02,03,12,13,23,012,013,023,123*/
char subalgebras[NUMBER][11]; 
struct m{ char *v; char value;}; 
struct m arr1[24]; 
// this is the array used for searching for a majority term,
// each member contains a string of three arguments for this function,
// like "012", and a character value of the function, like '2'
struct m arr2[36];
// this is the array used for searching for a wnu3 term,
// each member contains a string of three arguments for this function,
// like "002", and a character value of the function, like '2'
struct m arr3[12]; 
struct m arr7[10];
struct m arr4[22]; 
struct m newarr[48];	
struct m arr9[8];
struct m arr11[6];
struct m arr13[4];
struct m newarray[48];
//this is the array used for searching for terms p and q

char arr5[4096][12]; 
char arr6[1024][10]; 
char arr8[256][8];   
char arr10[64][6]; 
char arr12[16][4];  
char bincomm [DIM][6]; /* the matrix containig all binary commutative operations*/



/* prototypes of functions used*/

int check1( struct m *, int, int);
int check2( struct m *, int, int);
int check3( struct m *, int, int);
void add_one(  char *, char *);
int edge_number( char *);
int check_isomorphism(char *, char *);
int f_majority(int); /*the function searching for a majority term*/
void f_m(struct m *, int, int);
int compare( char *, char *);
void set_arr1(void);
void f_bin( char [][6]);
int f_bcommp( char *);
int f_par( int, int, int, char *);
void f_3( struct m *, int, int);
void g_3( struct m *, int, int);
int f_wnu3(int); /*the function searching for a wnu3 term*/
void set_arr2(void);
int backwards(struct m *, int, int);
int back_2(struct m *, int, int);
int back_3(struct m *, int, int);
void f_subalgebras(void); 
void set_arr3(void); 
int g_wnu3(int); /*the function searching for a wnu3 term*/
void set_arr4(void); 

void set_arr5(void); 
int check5( char *, int); 
void set12_arr2(int); 
void n_arr2 (void); 
void set_arr6(void);
int check6( char *, int); 
void n6_arr2(int);
void set_arr8(void);
int check8( char *, int); 
void n8_arr2(int);
void set_arr10(void);
void set_arr12(void);
int check10( char *, int);
int check12( char *, int); 
void n10_arr2(int);
void n12_arr2(int);
void set_newarray(void);
int my_terms(int);
void m_t(struct m *,int, int);
int new_back(struct m *, int, int);
int newcheck(struct m *, int, int);



void main()
{
	int i, n_edge1, n_edge2, iso, j, num_wnu3, k, aux, aux4, aux3; 

	int num_noniso = 0, auxiliary, aux2, num, num2, num3, num4; 

	/*initializing the array of structures*/
	for(i=0; i<NUMBER; i++)
	   {
		arr[i].isomorph='0';
		arr[i].majority='0';
		arr[i].wnu2='0';
		arr[i].wnu3= '0';
		for (j=0; j<16; j++)  (arr[i].description)[j] ='0';
		                         (arr[i].description)[16] ='\0';	
	     } 
	
	
	
	/*everythig is set to zero now*/

	
	/*generating all digraphs with four vertices*/

	   for(i=1; i<NUMBER; i++)
		add_one(arr[i-1].description, arr[i].description);

//  by this we generated all digraphs on four vertices, 
//  we assume edges are in this order:
//	   <0,0>, <0,1>, <0,2>, <0,3>, <1,0>, ...<3,3>
	
//  we search for non-isomorphic digraphs now, 
// coppies will have '1' and  originals '0' on the isomorphism flag
   
	
	for( i=1; i<NUMBER ; i++)  
		if( arr[i].isomorph == '0')
		{
		n_edge1 = edge_number(arr[i].description); /*number of edges in a digraph*/
		for( j=i+1; j<NUMBER ; j++) 
	    if( arr[j].isomorph =='0'){
			iso =0; 
		n_edge2 = edge_number( arr[j].description); 
		if( n_edge1 == n_edge2) iso = 
check_isomorphism(arr[i].description, arr[j].description); 
		if( iso ==1) arr[j].isomorph = '1';} 

		}

	/* counting the non_isomorphic ones*/

	for( i=0; i<NUMBER; i++) if( arr[i].isomorph == '0')  num_noniso ++; 

	printf (" \n\n non_isomorphic : %d\n", num_noniso); 

	
	

	/* we are setting subalgebras for non_isomorphic digraphs*/ 

	f_subalgebras(); 

	 printf("\n\n subalgebras set"); 
	 
	set_arr1();  /*initializing the array arr1 
                        in order to search for a  majority term*/

	
	for( i=0; i<NUMBER; i++)  
		if( arr[i].isomorph == '0')  arr[i].majority+= f_majority(i); 
		

	/* counting digraphs having majority term*/

	auxiliary =0; 

	for( i=0; i<NUMBER; i++) if( arr[i].majority == '1') auxiliary ++; 

	printf("\n\n a majority term : %d", auxiliary); 
	
	
// digraphs that are not isomorphic coppies and do not 
//	have a majority term are now examined for a wnu2 term operation, 
//	i.e a binary commutative operation
	


// this function will generate all binary commutative operations, 
//	that is initialize the matrix bincomm

	f_bin(bincomm); 

	
   /*now we test the digraps*/

	for(i=0; i<NUMBER; i++)
		if(arr[i].isomorph =='0' && arr[i].majority =='0' ) 
			arr[i].wnu2+=f_bcommp(arr[i].description); 
        
// all non_isomorphic digraphs having no majority function 
//	but having a wnu2 function now have wnu2 flag set to '1'

	
	auxiliary =0; 

	for(i=0; i<NUMBER; i++)
		if(arr[i].isomorph =='0' && arr[i].majority =='0' 
                && arr[i].wnu2 =='0' )
		{ auxiliary++; }

    printf("\n\n no majority neither wnu2 :%d", auxiliary);

	
	

// digraphs having neither a majority term nor a wnu2 term 
// should now be examined for a wnu3 term
// if there is a wnu3 term they are of bounded width

// this section of code is under comment for it does not work fast enough,
//	at least not on a single processor, 
//	so we had to find another way to calculate wnu3 terms
// anyhow it is presented here for the sake of uniformity
	 
//	set_arr2(); /* setting the array arr2 for a wnu3 term*/
//    auxiliary =0;
//	for(i=0; i<NUMBER; i++) 
//		if( arr[i].isomorph =='0' && arr[i].majority =='0' 
                 && arr[i].wnu2 =='0')
//		{ auxiliary++;  arr[i].wnu3 += f_wnu3(i); }
	                                                                              
    
//	auxiliary =0; 
//	for(i=0; i<NUMBER; i++) if( arr[i].wnu3=='1') auxiliary ++;

//    printf("\n\n no majority no wnu2 but having a wnu3 :%d", auxiliary);




	/*let us explain the idea for faster calculating wnu3 */

//it contains of the following steps:
//-first we identify two-element subalgebras of a digraph
//-then we form a matrix containing possible values of corresponding elements 
//(for example if {0,1} is a subalgebra then g(0,1,1) and g(1,1,0) can only
//have values 0 or 1, g being a wnu3 term we are looking for)
//-we copy a row of this matrix at the beginning of the array arr2 used for backtracking
//-they we start backtracking on the rest of this array
//-if the result is 0, that is no wnu3, we copy the next row of the matrix 
// to arr2 and try again
//-since the backtracking is now done on a shorter array it is much faster and 
//can be done on a single processor 

	/*that is how it is done here*/
	
	
	/*back to the code*/
  
	
/*digraphs with six two-element subalgebras */

//this means wnu3 term can only have one of two  possible values 
// on each of these triples of arguments:
//001,002,003,110,112,113,220,221,223,330,331,332
//in this case backtracking on the array arr2 should start from the 12th member, 
//so we have to change this limit from zero to twelve 
//in the function g_wnu3, and also in g_3 and back_3 before 
// executing this part of the code

//the function mentioned should look like this:

//	  int g_wnu3(int k){
//	  int i;
	
//	  for(i=12; i<36; i++) arr2[i].value ='a';
	 
//	  g_3(&arr2[12],12,k);  

//    if( arr2[35].value =='a') return 0; 

//     return 1; } /* g_wnu3*/


		

// void g_3(struct m *p, int i, int k){

 
//  int x=0, y=1; 
//	  while (!x && y) { y=0; 
//	                    (*p).value ='0';
//						x=check3(p, i, k);  
//                        if(!x) { (*p).value ='1';
//						      x=check3(p, i, k); }
//						if(!x) { (*p).value ='2';
//						      x=check3(p, i, k); }
//						if(!x) { (*p).value ='3';
//						      x=check3(p, i, k); }
//						if(!x) { (*p).value ='a';
//						if(i>12) y=back_3(p-1, i-1,k); } } /*while*/

//          if(x==1 && i<35) g_3(p+1,i+1,k); 
//  else return;} /* g_3*/


// the function  back_3 sets the next value on the given element of the array arr2 and 
//   checks whether it is ok, if so it returns 1.  If it is impossible to set this member
//   it will try to set the previuos one and so on..., 
//   it returns  0 if it is impossible to set values in any way,
//   i.e. from the start to the element given
  
// this function is being called by the function  g_3 to set the previous member of arr2, 
// once it is recognized that the current member can not get any value 
// considering previuos members

//	int back_3(struct m *p, int i, int k){

//	  int x=0, y=1;

	  /* p is a pointer to the member of the array arr2 which is to be set again,
	  i is the index of this member, 
	  k is the index of the digraph*/
	  
	

//  if((*p).value =='3' && i==12) return 0; /* can not go any further*/

//    while (x==0 && (*p).value<'3' ) {

//		  (*p).value+=1;
//		  x=check3(p,i,k);} /*while*/
//  if(x==1) return 1;
//	  else if (i==12) return 0; 
//	  else{ 
//                while( !x && y)
//				{ y=0; 
//		         (*p).value ='a';
//				 y = back_3(p-1, i-1, k);
//				 if(y){  (*p).value ='0';
//				         x=check3(p,i,k);
//						 if(!x) { (*p).value ='1';
//						 x=check3(p,i,k);}
//						 if(!x) { (*p).value ='2';
//						 x=check3(p,i,k);}
//						 if(!x) { (*p).value ='3';
//						 x=check3(p,i,k);} } 
//		      } /*while*/

//      	  if(x==1) return 1;
//		  return 0;} /*else*/
//          } /*back_3*/

  
			  
			  
			  /* back to the code*/
    auxiliary =0;

	set_arr5(); /*this is the matrix containing all possible values 
				    for 12 choices of arguments mentioned above*/
    num3=0; 
	num =0; num2=0; 
	for(i=0; i<NUMBER; i++) {
	if( arr[i].isomorph =='0' && arr[i].majority =='0' && arr[i].wnu2 =='0' 
		&& arr[i].wnu3=='0' && 
		(subalgebras[i][0]-'0'  +   subalgebras[i][1]-'0' 
                +  subalgebras[i][2]-'0'  +  
		subalgebras[i][3]-'0'  +  
		subalgebras[i][4]-'0'  +   subalgebras[i][5]-'0'  == 6   )) 
		for(j=0; j<DIM; j++){ aux2 =check5(arr5[j], i); 
    
// check5 function tests whether the values assigned to these 12 triples are ok
//	according to digraph being examined*/
//	if so, we copy this row of the matrix to arr2 (set12_arr2 function) 
//	which means the first twelve elements of the array are set, 
//	and then we call the function g_wnu3 to try to set the rest of the array
	   
	   if (aux2) { set12_arr2(j);  auxiliary = g_wnu3(i);
		  
				if (auxiliary) {arr[i].wnu3='1';  num++; break;}
						
	} } }/*for i*/  

// when executed this part of the code we found out 
// none of these digraphs had a wnu3 term, so we set '2' on the corresponding flags*/ 


    for(i=0; i<NUMBER; i++) 
    if( arr[i].isomorph =='0' && arr[i].majority =='0' && arr[i].wnu2 =='0' 
		&& arr[i].wnu3=='0' && 
		(subalgebras[i][0]-'0'  +   subalgebras[i][1]-'0'  
                +  subalgebras[i][2]-'0'  +  
		subalgebras[i][3]-'0'  +  
		subalgebras[i][4]-'0'  +   
                subalgebras[i][5]-'0'  == 6   )) arr[i].wnu3='2';


// now we do the same for digraphs with five two-element subalgebras
// in this case we form a matrix containig possible values for ten triples of arguments
// thereby we can shorten backtracking on arr2 for ten members, 
// assosiating their values row by row from this matrix
	  

// in this case the limit for the backtracking  in arr2 would be from the 10th member on, 
// so in the functions g_nu3, g_3 and back_3 
// every occurance of 12 should be replaced with 10

// also, there are six possible cases, i.e. choices of five two-element algebras
// out of six, 
// so for each case we should redefine a bit functions set_arr6, check6 and n6_arr2 
// and execute this code again
// for example if sublagebras are 01,02,03,12,13 this means that we should 
// set values from the corresponding matrix for triples 
// 001,002,003,110,112,113,220,221,330,331 so the function 
// set_arr6 generating this matrix
// would look like this 
	
//	void set_arr6(void){

//		int i0,i1,i2,i3,i4,i5,i6,i7,i8,i9, j; 

//        j=0;
//		for(i0=0; i0<2; i0++) /*01*/
//			for(i1=0; i1<3; i1+=2) /*02*/
//				for(i2=0; i2<4; i2+=3) /*03*/
//                 for(i3=0; i3<2; i3++) /*10*/
//					 for(i4=1; i4<3; i4++) /*12*/
//						 for(i5=1; i5<4; i5+=2) /*13*/
//                            for(i6=0; i6<3; i6+=2) /*02*/
//								for(i7=1; i7<3; i7++) /*12*/
//									for(i8=0; i8<4; i8+=3) /*03*/
//										for(i9=1; i9<4; i9+=2) /*13*/
//											
//												{arr6[j][0]=i0+'0';
//		                                         arr6[j][1]=i1+'0';
//												 arr6[j][2]=i2+'0';
//												 arr6[j][3]=i3+'0';
//												 arr6[j][4]=i4+'0';
//												 arr6[j][5]=i5+'0';
//												 arr6[j][6]=i6+'0';
//												 arr6[j][7]=i7+'0';
//												 arr6[j][8]=i8+'0';
//												 arr6[j][9]=i9+'0';
//												 j++;}
//	}/*set_arr6*/

	
// the function check6 should look like this:
	
//	int check6( char *p, int i){  int j,k; int pj,dj,tj,vj, pk,dk,tk,vk;  int rez=1; 

		

//		(arr7[0]).v = "001"; 
//		(arr7[1]).v = "002"; 
//		(arr7[2]).v = "003"; 
//		(arr7[3]).v = "110";
//		(arr7[4]).v = "112"; 
//		(arr7[5]).v = "113";
//		(arr7[6]).v = "220"; 
//		(arr7[7]).v = "221";
//		(arr7[8]).v = "330"; 
//		(arr7[9]).v = "331"; 
		
//		the rest of the function  doesn't change

// n6_arr2 function would have to look like this:

		
//	   void n6_arr2(int j) 
	   
//	   { int k; 

//		(arr2[0]).v = "001"; 
//		(arr2[1]).v = "002"; 
//		(arr2[2]).v = "003"; 
//		(arr2[3]).v = "110";
//		(arr2[4]).v = "112"; 
//		(arr2[5]).v = "113";
//		(arr2[6]).v = "220"; 
//		(arr2[7]).v = "221";
//		(arr2[8]).v = "330"; 
//		(arr2[9]).v = "331"; 

//		(arr2[10]).v = "223";
//		(arr2[11]).v = "332"; 


//		the rest of this function is the same in all cases

	/* back to the code*/
	
	 set_arr6(); /*this is setting the matrix mentioned*/
     num3=0; 
	 num =0; num2=0; 
	for(i=0; i<NUMBER; i++) {
	if( arr[i].isomorph =='0' && arr[i].majority =='0' && arr[i].wnu2 =='0'
		&& arr[i].wnu3=='0' && 
		(subalgebras[i][0]=='1'  &&  subalgebras[i][1]=='1'  
                &&  subalgebras[i][2]=='1' 
		&&  subalgebras[i][3]=='1'  &&
		subalgebras[i][4]=='1' &&   subalgebras[i][5]=='0'  )) 
		for(j=0; j<1024; j++){ aux2 =check6(arr6[j], i); 
	
// check6 function tests whether the values assigned to these 10 triples are ok
//	according to digraph being examined
	    
	if (aux2) { n6_arr2(j);  auxiliary = g_wnu3(i);
		  
						if (auxiliary) {arr[i].wnu3='1'; 
						num++; break;}
						
	} } }/*for i*/  

// when executed this part of the code we found out 
// none of these digraphs had a wnu3 term either, 
// so we set '2' on the corresponding flags
// in order to exclude them from further examination
	
	for(i=0; i<NUMBER; i++) 
    if( arr[i].isomorph =='0' && arr[i].majority =='0' && arr[i].wnu2 =='0' 
		&& arr[i].wnu3=='0' && 
	(subalgebras[i][0]-'0'  +   subalgebras[i][1]-'0'  +  subalgebras[i][2]-'0'  +  
	subalgebras[i][3]-'0'  +  
 	subalgebras[i][4]-'0'  +   subalgebras[i][5]-'0'  == 5  )) arr[i].wnu3='2'; 





//the same procedure for digraphs with four two-element subalgebras
// the backtracking is shortened for 8 members here
// there are 15 iterations in this case
// we change the limit in g_wnu3, g_3 and back_3 from 10 to 8 as explained above
// for each iteration we redefine functions   
// set_arr8, check8 and n8_arr2
//like  explained above,
//  that is depending on the subalgebras present

	set_arr8();

	 num3=0; 
     num =0; num2=0; num4=0;
	 for(i=0; i<NUMBER; i++) {
	 if( arr[i].isomorph =='0' && arr[i].majority =='0' && arr[i].wnu2 =='0' 
		 && arr[i].wnu3=='0' && 
		(subalgebras[i][0]=='0'  &&  subalgebras[i][1]=='1'  &&  
                subalgebras[i][2]=='1'  &&  
		subalgebras[i][3]=='1'  &&
		subalgebras[i][4]=='0' &&   subalgebras[i][5]=='1'  )) {num2++; 
		for(j=0; j<256; j++){ aux2 =check8(arr8[j], i); num4++; 
	   if (aux2) { n8_arr2(j);  auxiliary = g_wnu3(i); num3++;
		  
						if (auxiliary) {arr[i].wnu3='1'; 
					     printf("\n\n the index of the digraph is : %d", i); 
					     printf("\n\n the digraph is :%s", arr[i].description); 
					     num++; break;}
						
	} }} }/*for i*/  





  printf("\n\n  the number of digraphs having  wnu3 je %d", num); 

// the same procedure for digraphs with three two-element subalgebras
// the backtracking is shortened for 6 members here
// there are 20 iterations in this case
// we change the limit in g_wnu3, g_3 and back_3 from 8 to 6 as explained above
// for each iteration we redefine functions 
// set_arr10, check10 and n10_arr2 
// in the way explained above,
 // that is depending on the subalgebras present
 
	set_arr10();

	num3=0; 
	num =0; num2=0; num4=0;
	for(i=0; i<NUMBER; i++) {
	if( arr[i].isomorph =='0' && arr[i].majority =='0' && arr[i].wnu2 =='0' 
		&& arr[i].wnu3=='0' && 
		(subalgebras[i][0]=='0'  &&  subalgebras[i][1]=='0'  &&  
                subalgebras[i][2]=='0'  &&  
		subalgebras[i][3]=='1'  &&
		subalgebras[i][4]=='1' &&   subalgebras[i][5]=='1'  )) {num2++; 
		for(j=0; j<64; j++){ aux2 =check10(arr10[j], i); num4++;  
        if (aux2) { n10_arr2(j);  auxiliary = g_wnu3(i); num3++;
		  
					if (auxiliary) {arr[i].wnu3='1'; 
					printf("\n\n the index of the digraph is : %d", i); 
					printf("\n\n the digraph is :%s", arr[i].description); 
				     num++; break;}
						
		} }} }/*for i */  
 

  printf("\n\n  the number of digraphs having wnu3 is %d", num); 
  
// the same procedure for digraphs with two two-element subalgebras
// the backtracking is shortened for 4 members here
// there are 15 iterations in this case
//  we change the limit in g_wnu3, g_3 and back_3 from 6 to 4 as explained above
// for each iteration we redefine functions
// set_arr12, check12 and n12_arr2 
//in the way explained above, 
// that is depending on the subalgebras present
	
    set_arr12();

	num3=0; 
    num =0; num2=0; num4=0;
	for(i=0; i<NUMBER; i++) {
	if( arr[i].isomorph =='0' && arr[i].majority =='0' && arr[i].wnu2 =='0'
		&& arr[i].wnu3=='0' && 
		(subalgebras[i][0]=='1'  &&  subalgebras[i][1]=='1'  
                &&  subalgebras[i][2]=='0'  &&  
		subalgebras[i][3]=='0'  &&
		subalgebras[i][4]=='0' &&   subalgebras[i][5]=='0'  )) {num2++; 
		for(j=0; j<16; j++){ aux2 =check12(arr12[j], i);  num4++; 
       if (aux2) { n12_arr2(j);  auxiliary = g_wnu3(i); num3++;
		  
					if (auxiliary) {arr[i].wnu3='1';
						printf("\n\n the index of the digraph is : %d", i); 
					printf("\n\n the digraph is  :%s", arr[i].description); 
				    num++;break;} 	
						
	   } } }} /*for i*/  


  printf("\n\n  the number of digraphs having  wnu3 je %d", num); 

// the digraphs having a single two-element subalgebra 
//and the ones with no any are all examined in the same way,
//  none of them has a wnu3 term
// we just omitted these parts of the code for the sake of shortening the whole thing

//we found out the digraphs having a wnu3 term were the ones indexed by 856, 6654, 1246, 1502, 2014, 2270, 5598, 5784, 6110, 6650, 6849, 7127, 14302, 14846, 16350; here are only 15 digraphs, but that is because the remaining 14 were found to be dual in pairs to 14 of these. 

//for these digraphs we found wnu3 terms (the least ones regarding values) and used them when searching for terms p and q (wnu3 terms being q terms). As explained above in each of these cases we found corresponding p terms so as to satisfy the system (1).

//for the sake of simplicity we shall just provide the functions searching for terms $p$ and $q$ at the end of this code(it is named my_terms) because executing it requires setting the first 12 members of the backtracking array to be the values of the wnu3 term found for a particular digraph. The similar thing  has already been explained and coded earlier, when searching for a wnu3 term. Basically we just set the beginning of the backtracking array  for a digraph, say number 856,accoriding to its wnu3 term, and then call the function, like auxiliary = my_terms(856). In the result we obtain 1 for all these digraphs. 

 set_newarray(); 
 auxiliary = my_terms(856);  

  getchar();
  getchar(); 

}   /*main*/ 

// the function generates a new digraph by adding 
// a binary number one to the string with the adress p and 
// setting a new string on the adress q

void add_one( char *p, char *q){

	char *p1=p, *q1 =q; 
    while(*p1){
		*q1 = *p1;
		p1++; q1++;
	}
	 *q1 ='\0';
	 q1--;
	 while(*q1 =='1') q1 --;
	 *q1='1';
	 q1++;
	while( *q1 =='1') { *q1 ='0'; q1++;}
} 

// this function counts edges in the digraph on the adress p

int edge_number( char *p) 
{
	   int i=0; 
	   while(*p){ if( *p =='1') i++; p++;} 

	   return i;}


// this function checks whether the second digraph 
// is an isomorphic copy  of the first one

int check_isomorphism( char *p, char *q)
{
	   char aux[17]; int i,j, p0, p1, p2, p3;  
	   char *q1, *r1; int flag; 

	   char *perm[23];    /* this keeps all the permutations except for the identity one*/ 

	   perm[0] = "0132";
	   perm[1] = "0213";
	   perm[2] = "0231";
	   perm[3] = "0312";
	   perm[4] = "0321";
	   perm[5] = "1023";
	   perm[6] = "1032";
	   perm[7] = "1203";
	   perm[8] = "1230";
	   perm[9] = "1302";
	   perm[10] = "1320";
	   perm[11] = "2013";
	   perm[12] = "2031";
	   perm[13] = "2103";
	   perm[14] = "2130";
	   perm[15] = "2301";
	   perm[16] = "2310";
	   perm[17] = "3012";
	   perm[18] = "3021";
           perm[19] = "3102";
	   perm[20] = "3120";
	   perm[21] = "3201";
	   perm[22] = "3210";

	  
	   
	     for( i=0; i<23; i++) { 
		   for(j=0; j <16; j++) aux[j] ='0'; aux[16] = '\0'; 

		   
        p0 = perm[i][0] - '0';  /* the offset for each digit in this permutation */
		p1 = perm[i][1] - '0';
		p2 = perm[i][2] - '0';
		p3 = perm[i][3] - '0';

		if( p[0] =='1') aux[p0*4 + p0] ='1'; 
		if( p[1] =='1') aux[p0*4 + p1 ] ='1'; 
		if( p[2] =='1') aux[p0*4 + p2 ] ='1'; 
		if( p[3] =='1') aux[p0*4 + p3] ='1';
		if( p[4] =='1') aux[p1*4 + p0] ='1'; 
		if( p[5] =='1') aux[p1*4 + p1 ] ='1'; 
		if( p[6] =='1') aux[p1*4 + p2 ] ='1'; 
		if( p[7] =='1') aux[p1*4 + p3 ] ='1'; 
		if( p[8] =='1') aux[p2*4 + p0 ] ='1';
		if( p[9] =='1') aux[p2*4 + p1 ] ='1'; 
		if( p[10] =='1') aux[p2*4 + p2 ] ='1'; 
		if( p[11] =='1') aux[p2*4 + p3 ] ='1';
		if( p[12] =='1') aux[p3*4 + p0 ] ='1'; 
		if( p[13] =='1') aux[p3*4 + p1 ] ='1'; 
		if( p[14] =='1') aux[p3*4 + p2 ] ='1'; 
		if( p[15] =='1') aux[p3*4 + p3 ] ='1';

		/* the auxiliary array is now set*/ 

		/* the comparison*/ 

		r1 = aux; q1 = q;  flag = 0; 
		while (*q1) { if ( *q1 != *r1) flag =1; 
		q1++; r1++; } 

		if(!flag)  break;  /* they are isomorphic  */  
       } /* for i*/ 

	   if(!flag) return 1; 

	   return 0; 

		 }

// this function sets subalgebras for each digraph

   void f_subalgebras(void) {
	
	int i,j;
	
	for( i=0; i<NUMBER;i++){
		for( j=0;j<10;j++) subalgebras[i][j] = '0'; 
		subalgebras[i][10] = '\0'; }
	

	for( i=0; i<NUMBER ; i++) {

		if( arr[i].isomorph =='0'){  /* only for non-isomorphic digraphs*/
		/* subalgebra {0,1}*/

// whether there are edges from any node to these two only,
// and whether there are edges just from these two to any node

		for( j=0; j<4; j++) { if((arr[i].description)[4*j]=='1' && 
			(arr[i].description)[4*j +1]=='1'&&
			(arr[i].description)[4*j +2]=='0' &&
                        (arr[i].description)[4*j +3]=='0')
			subalgebras [i][0] ='1';}
		if( subalgebras [i][0] =='0') 
			for ( j=0; j<4; j++)  if(( arr[i].description)[j]=='1' 
				&& ( arr[i].description)[4+j]=='1' 
				&& ( arr[i].description)[8+j]=='0'
                                && (arr[i].description)[ 12+j]=='0') 
				subalgebras[i][0] ='1';

// whether these two nodes are the only ones that have a path of lenght 1 
//	   leading to them

	   if( subalgebras [i][0] =='0') 
		   if( ((arr[i].description)[0]=='1' || 
                          (arr[i].description)[4]=='1' 
			   || (arr[i].description)[8]=='1'|| 
			   (arr[i].description)[12]=='1') &&
                           ((arr[i].description)[1]=='1'
			   || (arr[i].description)[5]=='1'||
			   (arr[i].description)[9]=='1'||
                           (arr[i].description)[13]=='1') 
			   &&  (arr[i].description)[2]=='0'
			   && (arr[i].description)[6]=='0'&& 
                           (arr[i].description)[10]=='0' 
			   && (arr[i].description)[14]=='0' && 
			   (arr[i].description)[3]=='0' && 
                           (arr[i].description)[7]=='0'
			   && (arr[i].description)[11]=='0' && 
			   (arr[i].description)[15]=='0')

			   subalgebras [i][0] ='1';

// whether these two nodes are the only ones that have a path of lenght 1 
//	     leading from them

    if( subalgebras [i][0] =='0') 
		 if( ((arr[i].description)[0]=='1' || 
                           (arr[i].description)[1]=='1' 
			 || (arr[i].description)[2]=='1'|| 
			   (arr[i].description)[3]=='1') && 
                           ((arr[i].description)[4]=='1'
			   || (arr[i].description)[5]=='1'||
			   (arr[i].description)[6]=='1'|| 
                           (arr[i].description)[7]=='1') 
			   &&  (arr[i].description)[8]=='0'
			   && (arr[i].description)[9]=='0'&& 
                           (arr[i].description)[10]=='0' 
			   && (arr[i].description)[11]=='0' && 
			   (arr[i].description)[12]=='0' && 
                           (arr[i].description)[13]=='0'
			   && (arr[i].description)[14]=='0' && 
			   (arr[i].description)[15]=='0')

			   subalgebras [i][0] ='1';

// these are all the checks for the subalgebra  {0,1}

	/* subalgebra {0,2}*/
	/* ... the same as above */ 

	for( j=0; j<4; j++) { if((arr[i].description)[4*j]=='1' 
		&& (arr[i].description)[4*j +2]=='1'
		&&(arr[i].description)[4*j +1]=='0' &&
                (arr[i].description)[4*j +3]=='0')
		subalgebras [i][1] ='1';}
		if( subalgebras [i][1] =='0') 
			for ( j=0; j<4; j++)  
                                if(( arr[i].description)[j]=='1' 
				&& ( arr[i].description)[8+j]=='1' 
				&& ( arr[i].description)[4+j]=='0' 
                                && (arr[i].description)[ 12+j]=='0') 
				subalgebras[i][1] ='1';
 
	/* ...and the other way round*/
 if( subalgebras [i][1] =='0') 
		   if( ((arr[i].description)[0]=='1' || (arr[i].description)[4]=='1' 
			   || (arr[i].description)[8]=='1'|| 
			   (arr[i].description)[12]=='1') && ((arr[i].description)[2]=='1'
			   || (arr[i].description)[6]=='1'||
			   (arr[i].description)[10]=='1'|| (arr[i].description)[14]=='1') 
			   &&  (arr[i].description)[1]=='0'
			   && (arr[i].description)[5]=='0'&& (arr[i].description)[9]=='0' 
			   && (arr[i].description)[13]=='0' && 
			   (arr[i].description)[3]=='0' && (arr[i].description)[7]=='0'
			   && (arr[i].description)[11]=='0' && 
			   (arr[i].description)[15]=='0')

			   subalgebras [i][1] ='1';

     /* ...path length 1*/

    if( subalgebras [i][1] =='0') 
		 if( ((arr[i].description)[0]=='1' || (arr[i].description)[1]=='1' 
			 || (arr[i].description)[2]=='1'|| 
			   (arr[i].description)[3]=='1') && ((arr[i].description)[8]=='1'
			   || (arr[i].description)[9]=='1'||
			   (arr[i].description)[10]=='1'|| (arr[i].description)[11]=='1') 
			   &&  (arr[i].description)[4]=='0'
			   && (arr[i].description)[5]=='0'&& (arr[i].description)[6]=='0' 
			   && (arr[i].description)[7]=='0' && 
			   (arr[i].description)[12]=='0' && (arr[i].description)[13]=='0'
			   && (arr[i].description)[14]=='0' && 
			   (arr[i].description)[15]=='0')

			   subalgebras [i][1] ='1';
     /* ... length 1*/
	if( subalgebras [i][1] =='0') 
		 if( ((arr[i].description)[0]=='1' || (arr[i].description)[1]=='1' 
			 || (arr[i].description)[2]=='1'|| 
			   (arr[i].description)[3]=='1') && ((arr[i].description)[8]=='1'
			   || (arr[i].description)[9]=='1'||
			   (arr[i].description)[10]=='1'|| (arr[i].description)[11]=='1')
			   &&  (arr[i].description)[4]=='0'
			   && (arr[i].description)[5]=='0'&& (arr[i].description)[6]=='0' 
			   && (arr[i].description)[7]=='0' && 
			   (arr[i].description)[12]=='0' && (arr[i].description)[13]=='0'
			   && (arr[i].description)[14]=='0' && 
			   (arr[i].description)[15]=='0')
         subalgebras [i][1] ='1';

//  these are all the checks for the subalgebra {0,2}

	/* {0,3}*/
	/* ...*/ 
for( j=0; j<4; j++) { if((arr[i].description)[4*j]=='1' 
	&& (arr[i].description)[4*j +3]=='1'&&
			(arr[i].description)[4*j +1]=='0' && 
                        (arr[i].description)[4*j +2]=='0') 
			subalgebras [i][2] ='1';}
		if( subalgebras [i][2] =='0') 
			for ( j=0; j<4; j++)  if(( arr[i].description)[j]=='1' 
				&& ( arr[i].description)[12+j]=='1' 
				&& ( arr[i].description)[4+j]=='0' 
                                && (arr[i].description)[8+j]=='0') 
				subalgebras[i][2] ='1';

/* ...*/
 if( subalgebras [i][2] =='0') 
		   if( ((arr[i].description)[0]=='1' || (arr[i].description)[4]=='1' 
			   || (arr[i].description)[8]=='1'|| 
			   (arr[i].description)[12]=='1') && ((arr[i].description)[3]=='1'
			   || (arr[i].description)[7]=='1'||
			   (arr[i].description)[11]=='1'|| (arr[i].description)[15]=='1') 
			   &&  (arr[i].description)[1]=='0'
			   && (arr[i].description)[5]=='0'&& (arr[i].description)[9]=='0' 
			   && (arr[i].description)[13]=='0' && 
			   (arr[i].description)[2]=='0' && (arr[i].description)[6]=='0'
			   && (arr[i].description)[10]=='0' && 
			   (arr[i].description)[14]=='0')

			   subalgebras [i][2] ='1';

		/* ...*/

    if( subalgebras [i][2] =='0') 
		 if( ((arr[i].description)[0]=='1' || (arr[i].description)[1]=='1' 
			 || (arr[i].description)[2]=='1'|| 
			   (arr[i].description)[3]=='1') && ((arr[i].description)[12]=='1'
			   || (arr[i].description)[13]=='1'||
			   (arr[i].description)[14]=='1'|| (arr[i].description)[15]=='1') 
			   &&  (arr[i].description)[4]=='0'
			   && (arr[i].description)[5]=='0'&& (arr[i].description)[6]=='0' 
			   && (arr[i].description)[7]=='0' && 
			   (arr[i].description)[8]=='0' && (arr[i].description)[9]=='0'
			   && (arr[i].description)[10]=='0' && 
			   (arr[i].description)[11]=='0')

			   subalgebras [i][2] ='1';

    /* checked {0,3}*/

	/* subalgebra {1,2}*/
	/* ...*/ 

for( j=0; j<4; j++) { if((arr[i].description)[4*j +1]=='1' 
	&& (arr[i].description)[4*j +2]=='1'&&
			(arr[i].description)[4*j]=='0' &&
                        (arr[i].description)[4*j +3]=='0') 
			subalgebras [i][3] ='1';}
		if( subalgebras [i][3] =='0') 
			for ( j=0; j<4; j++)  if(( arr[i].description)[4+j]=='1' 
				&& ( arr[i].description)[8+j]=='1' 
				&& ( arr[i].description)[j]=='0' &&
                               (arr[i].description)[12+j]=='0')
				subalgebras[i][3] ='1';

   /*...*/
 if( subalgebras [i][3] =='0') 
		   if( ((arr[i].description)[1]=='1' || (arr[i].description)[5]=='1' 
			   || (arr[i].description)[9]=='1'|| 
			   (arr[i].description)[13]=='1') && ((arr[i].description)[2]=='1'
			   || (arr[i].description)[6]=='1'||
			   (arr[i].description)[10]=='1'|| (arr[i].description)[14]=='1') 
			   &&  (arr[i].description)[0]=='0'
			   && (arr[i].description)[4]=='0'&& (arr[i].description)[8]=='0' 
			   && (arr[i].description)[12]=='0' && 
			   (arr[i].description)[3]=='0' && (arr[i].description)[7]=='0'
			   && (arr[i].description)[11]=='0' && 
			   (arr[i].description)[15]=='0')

			   subalgebras [i][3] ='1';
   /* ...*/

    if( subalgebras [i][3] =='0') 
		 if( ((arr[i].description)[4]=='1' || (arr[i].description)[5]=='1' 
			 || (arr[i].description)[6]=='1'|| 
			   (arr[i].description)[7]=='1') && ((arr[i].description)[8]=='1'
			   || (arr[i].description)[9]=='1'||
			   (arr[i].description)[10]=='1'|| (arr[i].description)[11]=='1') 
			   &&  (arr[i].description)[0]=='0'
			   && (arr[i].description)[1]=='0'&& (arr[i].description)[2]=='0' 
			   && (arr[i].description)[3]=='0' && 
			   (arr[i].description)[12]=='0' && (arr[i].description)[13]=='0'
			   && (arr[i].description)[14]=='0' && 
			   (arr[i].description)[15]=='0')

			   subalgebras [i][3] ='1';
    /*  {1,2}*/

	/* subalgebra{1,3}*/
	/* ...*/ 

	for( j=0; j<4; j++) { if((arr[i].description)[4*j +1]=='1' 
		&& (arr[i].description)[4*j +3]=='1'&&
			(arr[i].description)[4*j]=='0' &&
                        (arr[i].description)[4*j +2]=='0') 
			subalgebras [i][4] ='1';}
		if( subalgebras [i][4] =='0') 
			for ( j=0; j<4; j++)  if(( arr[i].description)[4+j]=='1' 
				&& (arr[i].description)[12+j]=='1' 
				&& (arr[i].description)[j]=='0' &&
                               (arr[i].description)[8+j]=='0') 
				subalgebras[i][4] ='1';
    /* ... 1*/
 if( subalgebras [i][4] =='0') 
		   if( ((arr[i].description)[1]=='1' || (arr[i].description)[5]=='1' 
			   || (arr[i].description)[9]=='1'|| 
			   (arr[i].description)[13]=='1') && ((arr[i].description)[3]=='1'
			   || (arr[i].description)[7]=='1'||
			   (arr[i].description)[11]=='1'|| (arr[i].description)[15]=='1') 
			   &&  (arr[i].description)[0]=='0'
			   && (arr[i].description)[4]=='0'&& (arr[i].description)[8]=='0' 
			   && (arr[i].description)[12]=='0' && 
			   (arr[i].description)[2]=='0' && (arr[i].description)[6]=='0'
			   && (arr[i].description)[10]=='0' && 
			   (arr[i].description)[14]=='0')

			   subalgebras [i][4] ='1';
    /* ... 1*/

    if( subalgebras [i][4] =='0') 
		 if( ((arr[i].description)[4]=='1' || (arr[i].description)[5]=='1' 
			 || (arr[i].description)[6]=='1'|| 
			   (arr[i].description)[7]=='1') && ((arr[i].description)[12]=='1'
			   || (arr[i].description)[13]=='1'||
			   (arr[i].description)[14]=='1'|| (arr[i].description)[15]=='1') 
			   &&  (arr[i].description)[0]=='0'
			   && (arr[i].description)[1]=='0'&& (arr[i].description)[2]=='0' 
			   && (arr[i].description)[3]=='0' && 
			   (arr[i].description)[8]=='0' && (arr[i].description)[9]=='0'
			   && (arr[i].description)[10]=='0' && 
			   (arr[i].description)[11]=='0')

			   subalgebras [i][4] ='1';
    /* checked  {1,3}*/

	/* subalgebra {2,3}*/
	/* ...*/ 

	for( j=0; j<4; j++) { if((arr[i].description)[4*j +2]=='1' 
		&& (arr[i].description)[4*j +3]=='1'&&
			(arr[i].description)[4*j]=='0' && 
                        (arr[i].description)[4*j +1]=='0') 
			subalgebras [i][5] ='1';}
		if( subalgebras [i][5] =='0') 
			for ( j=0; j<4; j++)  if(( arr[i].description)[8+j]=='1' 
				&& (arr[i].description)[12+j]=='1' 
				&& (arr[i].description)[j]=='0' &&
                               (arr[i].description)[4+j]=='0') 
				subalgebras[i][5] ='1';
    /* ...1*/
 if( subalgebras [i][5] =='0') 
		   if( ((arr[i].description)[2]=='1' || (arr[i].description)[6]=='1' 
			   || (arr[i].description)[10]=='1'|| 
			   (arr[i].description)[14]=='1') && ((arr[i].description)[3]=='1'
			   || (arr[i].description)[7]=='1'||
			   (arr[i].description)[11]=='1'|| (arr[i].description)[15]=='1') 
			   &&  (arr[i].description)[0]=='0'
			   && (arr[i].description)[4]=='0'&& (arr[i].description)[8]=='0' 
			   && (arr[i].description)[12]=='0' && 
			   (arr[i].description)[1]=='0' && (arr[i].description)[5]=='0'
			   && (arr[i].description)[9]=='0' && 
			   (arr[i].description)[13]=='0')

			   subalgebras [i][5] ='1';
   /* ...1*/

    if( subalgebras [i][5] =='0') 
		 if( ((arr[i].description)[8]=='1' || (arr[i].description)[9]=='1' 
			 || (arr[i].description)[10]=='1'|| 
			   (arr[i].description)[11]=='1') && ((arr[i].description)[12]=='1'
			   || (arr[i].description)[13]=='1'||
			   (arr[i].description)[14]=='1'|| (arr[i].description)[15]=='1') 
			   &&  (arr[i].description)[0]=='0'
			   && (arr[i].description)[1]=='0'&& (arr[i].description)[2]=='0' 
			   && (arr[i].description)[3]=='0' && 
			   (arr[i].description)[4]=='0' && (arr[i].description)[5]=='0'
			   && (arr[i].description)[6]=='0' && 
			   (arr[i].description)[7]=='0')

			   subalgebras [i][5] ='1';
    /* checked  {2,3}*/

	/* subalgebra {0,1,2}*/
	/* ...*/ 
for( j=0; j<4; j++) { if((arr[i].description)[4*j]=='1' 
	&& (arr[i].description)[4*j +1]=='1'&& 
	(arr[i].description)[4*j +2]=='1' && (arr[i].description)[4*j +3]=='0') 
	subalgebras [i][6] ='1';}
		if( subalgebras [i][6] =='0') 
			for ( j=0; j<4; j++)  if(( arr[i].description)[j]=='1'
				&& (arr[i].description)[4+j]=='1' 
				&& (arr[i].description)[ 8+j]=='1' 
                                && (arr[i].description)[12+j]=='0')
				subalgebras[i][6] ='1';
    /* ... 1*/
 
if( subalgebras [i][6] =='0') 
		   if( ((arr[i].description)[0]=='1' || (arr[i].description)[4]=='1'
			   || (arr[i].description)[8]=='1'|| 
			   (arr[i].description)[12]=='1') && ((arr[i].description)[1]=='1'
			   || (arr[i].description)[5]=='1'||
			   (arr[i].description)[9]=='1'|| (arr[i].description)[13]=='1') 
			   && ((arr[i].description)[2]=='1'||
			   (arr[i].description)[6]=='1'||  (arr[i].description)[10]=='1'
			   || (arr[i].description)[14]=='1')
			    && (arr[i].description)[3]=='0'&& (arr[i].description)[7]=='0' 
				&& (arr[i].description)[11]=='0' && 
			   (arr[i].description)[15]=='0') 

			   subalgebras [i][6] ='1';
   /* ... 1*/
  if( subalgebras [i][6] =='0') 
		 if( ((arr[i].description)[0]=='1' || (arr[i].description)[1]=='1' 
			 || (arr[i].description)[2]=='1'|| 
			   (arr[i].description)[3]=='1') && ((arr[i].description)[4]=='1'
			   || (arr[i].description)[5]=='1'||
			   (arr[i].description)[6]=='1'|| (arr[i].description)[7]=='1')
			   && ((arr[i].description)[8]=='1'|| 
			   (arr[i].description)[9]=='1'||(arr[i].description)[10]=='1'
			   || (arr[i].description)[11]=='1')
			    && (arr[i].description)[12]=='0'&& (arr[i].description)[13]=='0' 
				&& (arr[i].description)[14]=='0' && 
			   (arr[i].description)[15]=='0') 

			   subalgebras [i][6] ='1';
    /* done {0,1,2}*/

	/* subalgebra {0,1,3}*/
	/* ...*/ 
   /* ...*/ 
for( j=0; j<4; j++) { if((arr[i].description)[4*j]=='1' 
	&& (arr[i].description)[4*j +1]=='1'&& 
	(arr[i].description)[4*j +3]=='1' && (arr[i].description)[4*j +2]=='0') 
	subalgebras [i][7] ='1';}
		if( subalgebras [i][7] =='0') 
			for ( j=0; j<4; j++)  if(( arr[i].description)[j]=='1' 
				&& (arr[i].description)[4+j]=='1' 
				&& (arr[i].description)[12+j]=='1' 
                                && (arr[i].description)[8+j]=='0')
				subalgebras[i][7] ='1';
    /* ...1*/
 
if( subalgebras [i][7] =='0') 
		   if( ((arr[i].description)[0]=='1' || (arr[i].description)[4]=='1' 
			   || (arr[i].description)[8]=='1'|| 
			   (arr[i].description)[12]=='1') && ((arr[i].description)[1]=='1'
			   || (arr[i].description)[5]=='1'||
			   (arr[i].description)[9]=='1'|| (arr[i].description)[13]=='1') 
			   && ((arr[i].description)[3]=='1'||
			   (arr[i].description)[7]=='1'||  (arr[i].description)[11]=='1'
			   || (arr[i].description)[15]=='1')
			    && (arr[i].description)[2]=='0'&& (arr[i].description)[6]=='0' 
				&& (arr[i].description)[10]=='0' && 
			   (arr[i].description)[14]=='0') 

			   subalgebras [i][7] ='1';
    /* ...1*/
  if( subalgebras [i][7] =='0') 
		 if( ((arr[i].description)[0]=='1' || (arr[i].description)[1]=='1'
			 || (arr[i].description)[2]=='1'|| 
			   (arr[i].description)[3]=='1') && ((arr[i].description)[4]=='1'
			   || (arr[i].description)[5]=='1'||
			   (arr[i].description)[6]=='1'|| (arr[i].description)[7]=='1') 
			   && ((arr[i].description)[12]=='1'|| 
			   (arr[i].description)[13]=='1'||(arr[i].description)[14]=='1'
			   || (arr[i].description)[15]=='1')
			    && (arr[i].description)[8]=='0'&& (arr[i].description)[9]=='0' 
				&& (arr[i].description)[10]=='0' && 
			   (arr[i].description)[11]=='0') 

			   subalgebras [i][7] ='1';
    /* done {0,1,3}*/

	/* subalgebra {0,2,3}*/
	/*...*/ 
    /*...*/ 
for( j=0; j<4; j++) { if((arr[i].description)[4*j]=='1' 
	&& (arr[i].description)[4*j +2]=='1'&& 
	(arr[i].description)[4*j +3]=='1' && (arr[i].description)[4*j +1]=='0')
	subalgebras [i][8] ='1';}
		if( subalgebras [i][8] =='0') 
			for ( j=0; j<4; j++)  if(( arr[i].description)[j]=='1' 
				&& (arr[i].description)[8+j]=='1' 
				&& (arr[i].description)[12+j]=='1' && 
                               (arr[i].description)[4+j]=='0') 
				subalgebras[i][8] ='1';
    /* ...1*/
 
if( subalgebras [i][8] =='0') 
		   if( ((arr[i].description)[0]=='1' || (arr[i].description)[4]=='1'
			   || (arr[i].description)[8]=='1'|| 
			   (arr[i].description)[12]=='1') && ((arr[i].description)[2]=='1'
			   || (arr[i].description)[6]=='1'||
			   (arr[i].description)[10]=='1'|| (arr[i].description)[14]=='1') 
			   && ((arr[i].description)[3]=='1'||
			   (arr[i].description)[7]=='1'||  (arr[i].description)[11]=='1'
			   || (arr[i].description)[15]=='1')
			    && (arr[i].description)[1]=='0'&& (arr[i].description)[5]=='0' 
				&& (arr[i].description)[9]=='0' && 
			   (arr[i].description)[13]=='0') 

			   subalgebras [i][8] ='1';
    /* ...1*/
  if( subalgebras [i][8] =='0') 
		 if( ((arr[i].description)[0]=='1' || (arr[i].description)[1]=='1' 
			 || (arr[i].description)[2]=='1'|| 
			   (arr[i].description)[3]=='1') && ((arr[i].description)[8]=='1'
			   || (arr[i].description)[9]=='1'||
			   (arr[i].description)[10]=='1'|| (arr[i].description)[11]=='1') 
			   && ((arr[i].description)[12]=='1'|| 
			   (arr[i].description)[13]=='1'||(arr[i].description)[14]=='1'
			   || (arr[i].description)[15]=='1')
			    && (arr[i].description)[4]=='0'&& (arr[i].description)[5]=='0' 
				&& (arr[i].description)[6]=='0' && 
			   (arr[i].description)[7]=='0') 

			   subalgebras [i][8] ='1';
    /*done {0,2,3}*/

	/*subalgebra{1,2,3}*/
	/* ...*/ 
    /*...*/

for( j=0; j<4; j++) { if((arr[i].description)[4*j+1]=='1' 
	&& (arr[i].description)[4*j +2]=='1'&& 
	(arr[i].description)[4*j +3]=='1' && (arr[i].description)[4*j]=='0') 
	subalgebras [i][9] ='1';}
		if( subalgebras [i][9] =='0') 
			for ( j=0; j<4; j++)  if(( arr[i].description)[4+j]=='1' 
				&& (arr[i].description)[8+j]=='1' 
				&& (arr[i].description)[12+j]=='1' && 
                               (arr[i].description)[j]=='0') 
				subalgebras[i][9] ='1';
   /* ... 1*/
 
if( subalgebras [i][9] =='0') 
		   if( ((arr[i].description)[1]=='1' || (arr[i].description)[5]=='1' 
			   || (arr[i].description)[9]=='1'|| 
			   (arr[i].description)[13]=='1') && ((arr[i].description)[2]=='1'
			   || (arr[i].description)[6]=='1'||
			   (arr[i].description)[10]=='1'|| (arr[i].description)[14]=='1') 
			   && ((arr[i].description)[3]=='1'||
			   (arr[i].description)[7]=='1'||  (arr[i].description)[11]=='1'
			   || (arr[i].description)[15]=='1')
			    && (arr[i].description)[0]=='0'&& (arr[i].description)[4]=='0' 
				&& (arr[i].description)[8]=='0' && 
			   (arr[i].description)[12]=='0') 

			   subalgebras [i][9] ='1';
    /* ... 1*/
  if( subalgebras [i][9] =='0') 
		 if( ((arr[i].description)[4]=='1' || (arr[i].description)[5]=='1' 
			 || (arr[i].description)[6]=='1'|| 
			   (arr[i].description)[7]=='1') && ((arr[i].description)[8]=='1'
			   || (arr[i].description)[9]=='1'||
			   (arr[i].description)[10]=='1'|| (arr[i].description)[11]=='1') 
			   && ((arr[i].description)[12]=='1'|| 
			   (arr[i].description)[13]=='1'||(arr[i].description)[14]=='1'
			   || (arr[i].description)[15]=='1')
			    && (arr[i].description)[0]=='0'&& (arr[i].description)[1]=='0' 
				&& (arr[i].description)[2]=='0' && 
			   (arr[i].description)[3]=='0') 

			   subalgebras [i][9] ='1';
    /* done {1,2,3}*/

   /* we check the intersections of three element podalgebras*/
  if ( subalgebras [i][6] =='1' && subalgebras [i][7] =='1') subalgebras [i][0] ='1';
  if ( subalgebras [i][6] =='1' && subalgebras [i][8] =='1') subalgebras [i][1] ='1';
  if ( subalgebras [i][6] =='1' && subalgebras [i][9] =='1') subalgebras [i][3] ='1';
  if ( subalgebras [i][7] =='1' && subalgebras [i][8] =='1') subalgebras [i][2] ='1';
  if ( subalgebras [i][7] =='1' && subalgebras [i][9] =='1') subalgebras [i][4] ='1';
  if ( subalgebras [i][8] =='1' && subalgebras [i][9] =='1') subalgebras [i][5] ='1';
 
} } /* for i*/

} /*f_subalgebras*/

// this function sets values of the extern array arr1
// before searching for a majority term 

void set_arr1(void){
	int i;
	for(i=0; i<24; i++) arr1[i].value ='a'; 
	
	/* value 'a' would be replaced with '0','1','2','3' as backtracking progresses */

	/*we set triples of arguments for a majority term*/

	arr1[0].v="012";
	arr1[1].v="013";
	arr1[2].v="021";
	arr1[3].v="023";
	arr1[4].v="031";
	arr1[5].v="032";
	arr1[6].v="102";
	arr1[7].v="103";
	arr1[8].v="120";
	arr1[9].v="123";
	arr1[10].v="130";
	arr1[11].v="132";
	arr1[12].v="201";
	arr1[13].v="203";
	arr1[14].v="210";
	arr1[15].v="213";
	arr1[16].v="230";
	arr1[17].v="231";
	arr1[18].v="301";
	arr1[19].v="302";
	arr1[20].v="310";
	arr1[21].v="312";
	arr1[22].v="320";
	arr1[23].v="321";
} /* set_arr1*/   


//  this function searches for a majority term

int f_majority(int k){
	int i; 
	for(i=0; i<24; i++) arr1[i].value = 'a';

	f_m(arr1, 0,k);
// this function will set the array arr1 in case there is a majority term

	if(arr1[23].value =='a') return 0; 
	
	return 1; } /* f_majority*/

//  this function compares two strings, it returns 1 if equal, 0 if not

int compare( char *p, char *q){

	char *p1 =p, *q1 =q; 
	int r=1;
	while(*p1) {if (*p1 != *q1) r=0; p1++,q1++;}
	return r;} /*compare*/ 

// this function sets values of tha array arr1 for a given digraph, 
// if there is a majority term the whole array will be set, 
// otherwise some of its members at the end will be left with the value 'a'

// the function f_majority tests the last member of tha array arr1 
// after this function has finished

void f_m(struct m *p, int i, int k){

	/* p is a pointer to the member of the array arr1 being currently set, 
	i is the index of this member,  k is the index of the digraph*/

	int x=0, y=1;
	while(!x && y) { y=0;
	                 (*p).value ='0';
					 x=check1(p,i,k); /* whether the value is ok*/ 
					 if(!x) { (*p).value ='1';
					 x=check1(p,i,k); }
					 if(!x) { (*p).value ='2';
					 x=check1(p,i,k); }
					 if(!x) { (*p).value ='3';
					 x=check1(p,i,k); }
					 if(!x) { (*p).value ='a';
					 if (i>0) y= backwards(p-1,i-1,k); }} /* while*/ 
	         if(x==1 && i< 23) f_m(p+1,i+1,k); 
			 else return; } /*f_m*/ 

// this function sets the next value on the member of the array arr1 
// and then checks whether it is possible, 
// if so it returns one, if not 0*

//  this function is called by the function f_m 
// for the previous member of tha array arr1, 
//once it is recognized that the current member 
// can not be set considering all previous


int backwards( struct m *p, int i, int k){
	
	/* p is a pointer to the member of the array arr1 being currently set, 
	i is the index oof this member,  k is the index of the digraph*/
	
	int x=0, y=1; 
	

	if((*p).value=='3' && i==0) return 0; /* can not go any further*/

	while(x==0 && (*p).value <'3') { (*p).value+=1; 
	                                     x=check1(p,i,k);} /*while*/
	if(x==1) return 1;
	else if (i==0) return 0; 
	else{ 
		 while(!x && y) { 
			               y=0;
						   (*p).value ='a';
						   y= backwards(p-1,i-1,k);
						   if(y) { (*p).value ='0';
						           x=check1(p,i,k); 
								   if(!x) { (*p).value ='1';
					                       x=check1(p,i,k); } 
						           
                                   if(!x) { (*p).value ='2';
					                       x=check1(p,i,k); } 
								   if(!x) { (*p).value ='3';
								   x=check1(p,i,k); } }/* if y*/ 
		 } /*while*/

		 if(x==1) return 1; 
		 return 0;}  } /*backwards*/

// the function check1 tests whether the value being just set
// in the array arr1 is ok considering previous members and subalgebras


int check1(struct m *p, int i, int k){ 
	
	/* p is a pointer to the current(the last one being set) 
	member of the array arr1, i is the index of this member, 
	k is the index of the digraph*/
										
										

	int prva, druga, treca, pj, dj, tj, vred, vj; 
	int i1, rez=1; 
	int j;

	prva=((*p).v)[0]-'0'; 
	druga=((*p).v)[1]-'0'; 
	treca=((*p).v)[2]-'0';
	vred = (*p).value -'0'; 
	
	
	/* we check subalgebras first*/

	if((subalgebras[k])[6] =='1' &&
	(compare((*p).v, "012") ==1 || compare((*p).v, "021") ==1 
	|| compare((*p).v, "102") ==1
	|| compare((*p).v, "120") ==1 || compare((*p).v, "201") ==1 
	|| compare((*p).v, "210") ==1)
	&& (*p).value =='3')

	{rez =0; return 0;}

	if((subalgebras[k])[7] =='1' &&
	(compare((*p).v, "013") ==1 || compare((*p).v, "031") ==1 
	|| compare((*p).v, "103") ==1
	|| compare((*p).v, "130") ==1 || compare((*p).v, "301") ==1 
	|| compare((*p).v, "310") ==1)
	&& (*p).value =='2')

	{rez =0; return 0;}

	if((subalgebras[k])[8] =='1' &&
	(compare((*p).v, "023") ==1 || compare((*p).v, "032") ==1 
	|| compare((*p).v, "203") ==1
	|| compare((*p).v, "230") ==1 || compare((*p).v, "302") ==1 
	|| compare((*p).v, "320") ==1)
	&& (*p).value =='1')

	{rez =0; return 0;}

	if((subalgebras[k])[9] =='1' &&
	(compare((*p).v, "123") ==1 || compare((*p).v, "132") ==1
	|| compare((*p).v, "213") ==1
	|| compare((*p).v, "231") ==1 || compare((*p).v, "312") ==1 
	|| compare((*p).v, "321") ==1)
	&& (*p).value =='0')

	{rez = 0; return 0;}

	/* subalgebras ok */

	/*checking the relation*/

	for( i1=0; i1<4; i1++){

		if( (arr[k].description)[4*i1 +prva] =='1' 
			&& (arr[k].description)[4*i1 +druga] =='1' 
		&& ( (arr[k].description)[treca] =='1' 
		||  (arr[k].description)[4+treca] =='1'
		||(arr[k].description)[8+treca] =='1'
		|| (arr[k].description)[12+treca] =='1') 
		&& (arr[k].description)[4*i1+vred] =='0'){rez =0; break;}

		if( (arr[k].description)[4*i1+prva] =='1' 
			&& (arr[k].description)[4*i1+treca] =='1' 
		&& ( (arr[k].description)[druga] =='1' 
		||  (arr[k].description)[4+druga] =='1'
		||(arr[k].description)[8+druga] =='1'
		|| (arr[k].description)[12+druga] =='1') 
		&& (arr[k].description)[4*i1+vred] =='0'){rez =0; break;}

		if( (arr[k].description)[4*i1 +druga] =='1' 
			&& (arr[k].description)[4*i1+treca] =='1' 
		&& ( (arr[k].description)[prva] =='1'
		||  (arr[k].description)[4+prva] =='1'
		||(arr[k].description)[8+prva] =='1'
		|| (arr[k].description)[12+prva] =='1') 
		&& (arr[k].description)[4*i1+vred] =='0'){rez =0; break;} }/*for*/


	  if(rez ==0) return 0; 

	  

	  for( i1=0; i1<4; i1++){

    if( (arr[k].description)[4*prva +i1] =='1' 
		&& (arr[k].description)[4*druga +i1] =='1'
	&&( (arr[k].description)[4*treca] =='1'
	|| (arr[k].description)[4*treca +1] =='1'
	||(arr[k].description)[4*treca +2] =='1'
	|| (arr[k].description)[4*treca +3] =='1') 
	 && (arr[k].description)[4*vred +i1] =='0') {rez =0; break;} 
	
	if( (arr[k].description)[4*prva +i1] =='1' 
		&& (arr[k].description)[4*treca +i1] =='1'
	&&( (arr[k].description)[4*druga] =='1'
	|| (arr[k].description)[4*druga +1] =='1'
	||(arr[k].description)[4*druga +2] =='1'
	|| (arr[k].description)[4*druga +3] =='1') 
	 && (arr[k].description)[4*vred +i1] =='0') {rez =0;break;} 

	if( (arr[k].description)[4*druga +i1] =='1' 
	&& (arr[k].description)[4*treca +i1] =='1'&&
	( (arr[k].description)[4*prva] =='1'
	|| (arr[k].description)[4*prva +1] =='1'||
	(arr[k].description)[4*prva +2] =='1'
	|| (arr[k].description)[4*prva +3] =='1') 
	&& (arr[k].description)[4*vred +i1] =='0') {rez =0; break;} }/*for*/

	  if(rez ==0) return 0; 

     /* checking previously set members of this array */



for( j=1; j<=i; j++){
	pj = ((*(p-j)).v)[0]-'0'; 
	dj = ((*(p-j)).v)[1]-'0'; 
	tj = ((*(p-j)).v)[2]-'0'; 
	vj = (*(p-j)).value -'0'; 
	

	if((arr[k].description)[pj*4 +prva] =='1' 
		&& (arr[k].description)[dj*4 +druga] =='1' 
	&& (arr[k].description)[tj*4 +treca] =='1' &&
	(arr[k].description)[vj*4+vred] =='0') { rez =0; break;}

	
	if((arr[k].description)[prva*4 +pj] =='1' 
		&& (arr[k].description)[druga*4 +dj] =='1' 
	&& (arr[k].description)[treca*4+tj] =='1' &&
	(arr[k].description)[vred*4+vj] =='0') { rez =0; break;}} /*for*/

if(rez ==0) return 0; 
 return 1; 
  } /*check1*/

//  the function f_bin sets the matrix of all binary commutative operations
//  (wnu2 term)
//  the order of values in each row is: 
//  <0,1> ,<0,2>, <0,3>, <1,2> , <1,3>, <2,3>

  void f_bin(char m[][6]){
	  int i0, i1,i2,i3, i4, i5,j;
	  j=0;

	  for(i0=0;i0<4; i0++)
		  for(i1=0;i1<4; i1++)
			  for(i2=0;i2<4; i2++)
				  for(i3=0;i3<4; i3++)
					  for(i4=0;i4<4; i4++)
						  for(i5=0;i5<4; i5++){
							  m[j][0]=i0+'0';
							   m[j][1]=i1+'0';
							    m[j][2]=i2+'0';
								 m[j][3]=i3+'0';
								  m[j][4]=i4+'0';
								   m[j][5]=i5+'0';
								   j++;}
  } /*f_bin*/

// this function tests whether the digraph with the address p has a  wnu2 term ,
//  it returnes 1 for yes, 0 for no
  int f_bcommp(char *p){
	  int x1=0, x2=0, x3=0, x4=0, x5=0, x6=0, x7=0, x8=0, x9=0, x10=0; 
	  int i;

	  for(i=0; i<DIM; i++){  
		  x1=f_par(0,0,i,p);
		  if(x1==0) continue;  /*a new operation from the matrix*/
		  x2=f_par(1,1,i,p);
		  if(x2==0) continue;
		  x3=f_par(2,2,i,p);
		  if(x3==0) continue;
		  x4=f_par(3,3,i,p);
		  if(x4==0) continue;
		  x5=f_par(0,1,i,p);
		  if(x5==0) continue;
		  x6=f_par(0,2,i,p);
		  if(x6==0) continue;
		  x7=f_par(0,3,i,p);
		  if(x7==0) continue;
		  x8=f_par(1,2,i,p);
		  if(x8==0) continue;
		  x9=f_par(1,3,i,p);
		  if(x9==0) continue;
		  x10=f_par(2,3,i,p);
		  if(x10==0) continue;
		  if( x1 && x2 && x3 && x4 && x5 && x6 && x7
		  && x8 && x9 && x10) break;} /*for*/

	  if( x1 && x2 && x3 && x4 && x5 && x6 && x7
		  && x8 && x9 && x10) return 1; 

	  return 0; } /* f_bcommp*/

 //this function checks whether the value of the operation with index i  
//  on the pair  <x,y> is compatible with the digraph having the address p

  int f_par( int x, int y, int i, char *p){

	  

	  int value; /*the value of the operation*/
	  if( x==y) value =x;
	  else if ( x==0 && y==1) value = bincomm[i][0] -'0';
	  else if ( x==0 && y==2) value = bincomm[i][1] -'0';
	  else if ( x==0 && y==3) value = bincomm[i][2] -'0';
	  else if ( x==1 && y==2) value = bincomm[i][3] -'0';
	  else if ( x==1 && y==3) value = bincomm[i][4] -'0';
	  else if ( x==2 && y==3) value = bincomm[i][5] -'0';
	 
	  /*we test the relation regarding this node*/

	  /* 0->x*/ 
	    if(p[x]=='1'){ if(x!=y && p[y] =='1' && p[value] =='0') 
			return 0; /* no compatibility*/

		               if(p[4+y] =='1' && p[4*(bincomm[i][0]-'0') +value] =='0') 
						   return 0; 
					   if(p[8+y] =='1' && p[4*(bincomm[i][1]-'0') +value] =='0') 
						   return 0; 
					   if(p[12+y] =='1' && p[4*(bincomm[i][2]-'0') +value] =='0')
						   return 0; 

		               } 

   /* 1->x*/
		if(p[4+x]=='1') { if(p[y] =='1' && p[4*(bincomm[i][0]-'0') +value] =='0') 
			return 0;
		                  if(x!=y && p[4+y] =='1' && p[4+value] =='0') 
							  return 0;
						  if(p[8+y] =='1' && p[4*(bincomm[i][3]-'0') +value] =='0') 
							  return 0;
						  if(p[12+y] =='1' && p[4*(bincomm[i][4]-'0') +value] =='0') 
							  return 0;

		                }
/* 2->x*/

     if(p[8+x]=='1') { if(p[y] =='1' && p[4*(bincomm[i][1]-'0') +value] =='0') 
		 return 0;
	                   if(p[4+y] =='1' && p[4*(bincomm[i][3]-'0') +value] =='0') 
						   return 0; 
					   if(x!=y && p[8+y] =='1' && p[8+value] =='0') 
						   return 0;
					   if(p[12+y] =='1' && p[4*(bincomm[i][5]-'0') +value] =='0') 
						   return 0;

	                  }
	 /*3->x*/

     if(p[12+x]=='1') {  if(p[y] =='1' && p[4*(bincomm[i][2]-'0') +value] =='0')
		 return 0;
	                     if(p[4+y] =='1' && p[4*(bincomm[i][4]-'0') +value] =='0') 
							 return 0; 
						 if(p[8+y] =='1' && p[4*(bincomm[i][5]-'0') +value] =='0')
							 return 0;
						 if(x!=y && p[12+y] =='1' && p[12+value] =='0') 
							 return 0;
	                   }

	
	 /*x->0*/

	  if(p[4*x]=='1') {  if(x!=y && p[4*y] =='1' && p[4*value] =='0')
		  return 0;
	                     if(p[4*y+1] =='1' && p[4*value + (bincomm[i][0]-'0')] =='0') 
							 return 0;
						 if(p[4*y+2] =='1' && p[4*value + (bincomm[i][1]-'0')] =='0') 
							 return 0;
						 if(p[4*y+3] =='1' && p[4*value + (bincomm[i][2]-'0')] =='0') 
							 return 0;
	                     }

	  /*x->1*/

     if(p[4*x +1]=='1') { if(p[4*y] =='1' && p[4*value + (bincomm[i][0]-'0')] =='0')
      return 0; 
	                      if(x!=y && p[4*y+1] =='1' && p[4*value+1] =='0') 
	                      return 0;
						  if(p[4*y+2] =='1' && p[4*value + (bincomm[i][3]-'0')] =='0')
						   return 0;
						  if(p[4*y+3] =='1' && p[4*value + (bincomm[i][4]-'0')] =='0')
						   return 0;
	 }
	 /*x->2*/

    if(p[4*x +2]=='1') { if(p[4*y] =='1' && p[4*value + (bincomm[i][1]-'0')] =='0') 
		return 0; 
	                     if(p[4*y+1] =='1' && p[4*value + (bincomm[i][3]-'0')] =='0') 
							 return 0;
                         if(x!=y && p[4*y+2] =='1' && p[4*value+2] =='0') 
							 return 0;
						  if(p[4*y+3] =='1' && p[4*value + (bincomm[i][5]-'0')] =='0') 
							  return 0;
	}
	/*x->3*/

	 if(p[4*x +3]=='1') {  if(p[4*y] =='1' && p[4*value + (bincomm[i][2]-'0')] =='0') 
		 return 0; 
	                       if(p[4*y+1] =='1' && p[4*value + (bincomm[i][4]-'0')] =='0') 
							   return 0;
						   if(p[4*y+2] =='1' && p[4*value + (bincomm[i][5]-'0')] =='0')
							   return 0;
						   if(x!=y && p[4*y+3] =='1' && p[4*value+3] =='0') 
							   return 0;
	 }

	 /*ok*/

	 return 1;

  } /*f_par*/

// this function sets values in the external array arr2 used for a wnu3 term
  
  void set_arr2(void){
	  int i; 
	  for(i=0; i<36; i++) arr2[i].value ='a'; 
	  

	  arr2[0].v = "001";
	  arr2[1].v = "002";
	  arr2[2].v = "003";
	  arr2[3].v = "110";
	  arr2[4].v = "112";
	  arr2[5].v = "113";
	  arr2[6].v = "220";
	  arr2[7].v = "221";
	  arr2[8].v = "223";
	  arr2[9].v = "330";
	  arr2[10].v = "331";
	  arr2[11].v = "332";
	  arr2[12].v = "012";
	  arr2[13].v = "013";
	  arr2[14].v = "021";
	  arr2[15].v = "023";
	  arr2[16].v = "031";
	  arr2[17].v = "032";
	  arr2[18].v = "102";
	  arr2[19].v = "103";
	  arr2[20].v = "120";
	  arr2[21].v = "123";
	  arr2[22].v = "130";
	  arr2[23].v = "132";
	  arr2[24].v = "201";
	  arr2[25].v = "203";
	  arr2[26].v = "210";
	  arr2[27].v = "213";
	  arr2[28].v = "230";
	  arr2[29].v = "231";
	  arr2[30].v = "301";
	  arr2[31].v = "302";
	  arr2[32].v = "310";
	  arr2[33].v = "312";
	  arr2[34].v = "320";
	  arr2[35].v = "321";
  } /* set_arr2*/


//  this function searches for a wnu3 term
// everything explained above considering majority term applies here too

  int f_wnu3(int k){
	  int i;
	
	  for(i=0; i<36; i++) arr2[i].value ='a';
	 
	  f_3(arr2, 0,k); 

      if( arr2[35].value =='a') return 0; 

	  return 1; } /* f_wnu3*/


  void f_3(struct m *p, int i, int k){



	  int x=0, y=1; 
	  while (!x && y) { y=0; 
	                    (*p).value ='0';
						x=check2(p, i, k); 

						if(!x) { (*p).value ='1';
						      x=check2(p, i, k); }
						if(!x) { (*p).value ='2';
						      x=check2(p, i, k); }
						if(!x) { (*p).value ='3';
						      x=check2(p, i, k); }
						if(!x) { (*p).value ='a';
						if(i>0) y=back_2(p-1, i-1,k); } } /*while*/

	  if(x==1 && i<35) f_3(p+1,i+1,k); 
	  else return;} /* f_3*/

  

  int back_2(struct m *p, int i, int k){

	  int x=0, y=1;

	  

	  if((*p).value =='3' && i==0) return 0; 

	  while (x==0 && (*p).value<'3' ) {

		  (*p).value+=1;
		  x=check2(p,i,k);} /*while*/
	  if(x==1) return 1;
	  else if (i==0) return 0; 
	  else{ 

		  while( !x && y){ y=0; 
		         (*p).value ='a';
				 y = back_2(p-1, i-1, k);
				 if(y){  (*p).value ='0';
				         x=check2(p,i,k);
						 if(!x) { (*p).value ='1';
						 x=check2(p,i,k);}
						 if(!x) { (*p).value ='2';
						 x=check2(p,i,k);}
						 if(!x) { (*p).value ='3';
						 x=check2(p,i,k);} } 
		  } /*while*/

		  if(x==1) return 1;
		  return 0;} /*else*/
                } /*back_2*/



  int check2(struct m *p, int i, int k){



	  int prva, druga, treca, pj, dj, tj, vred, vj, i1, rez, j;
	  rez =1;
	  prva = ((*p).v)[0] -'0';
	  druga = ((*p).v)[1] -'0';
	  treca = ((*p).v)[2] -'0';  

	  vred = (*p).value -'0';  

	 

	  if(i==0 && vred!=0 && vred!=1 && (subalgebras[k][0]=='1')) return 0; 
	  if(i==1 && vred!=0 && vred!=2 && (subalgebras[k][1]=='1')) return 0;
	  if(i==2 && vred!=0 && vred!=3 && (subalgebras[k][2]=='1')) return 0; 
	  if(i==3 && vred!=0 && vred!=1 && (subalgebras[k][0]=='1')) return 0; 
	  if(i==4 && vred!=1 && vred!=2 && (subalgebras[k][3]=='1')) return 0; 
	  if(i==5 && vred!=1 && vred!=3 && (subalgebras[k][4]=='1')) return 0; 
	  if(i==6 && vred!=0 && vred!=2 && (subalgebras[k][1]=='1')) return 0;
	  if(i==7 && vred!=1 && vred!=2 && (subalgebras[k][3]=='1')) return 0; 
	  if(i==8 && vred!=2 && vred!=3 && (subalgebras[k][5]=='1')) return 0; 
	  if(i==9 && vred!=0 && vred!=3 && (subalgebras[k][2]=='1')) return 0; 
	  if(i==10 && vred!=3 && vred!=1 && (subalgebras[k][4]=='1')) return 0; 
	  if(i==11 && vred!=3 && vred!=2 && (subalgebras[k][5]=='1')) return 0; 
	  if((i==12 || i==14 || i==18 || i==20 || i==24 || i==26) && vred ==3 &&
		  subalgebras[k][6]=='1') return 0;
	  if((i==13 || i==16 || i==19 || i==22 || i==30 || i==32) && vred ==2 &&
		  subalgebras[k][7]=='1') return 0;
	  if((i==15 || i==17 || i==25 || i==28 || i==31 || i==34) && vred ==1 &&
		  subalgebras[k][8]=='1') return 0;
	  if((i==21 || i==23 || i==27 || i==29 || i==33 || i==35) && vred ==0 &&
		  subalgebras[k][9]=='1') return 0;

	 

	  /* 0->...*/
	  if((arr[k].description)[prva] == '1' && 
		  (arr[k].description)[druga] == '1' && 
     (arr[k].description)[treca] == '1' 
	 && (arr[k].description)[vred] == '0') return 0; 

       /* 1->...*/

      if((arr[k].description)[4+prva] == '1' && 
		  (arr[k].description)[4+druga] == '1' && 
     (arr[k].description)[4+treca] == '1' 
	 && (arr[k].description)[4+vred] == '0') return 0;

	  if((arr[k].description)[8+prva] == '1' && 
		  (arr[k].description)[8+druga] == '1' && 
     (arr[k].description)[8+treca] == '1' && 
	 (arr[k].description)[8+vred] == '0') return 0;

	  if((arr[k].description)[12+prva] == '1' && 
		  (arr[k].description)[12+druga] == '1' && 
     (arr[k].description)[12+treca] == '1' &&
	 (arr[k].description)[12+vred] == '0') return 0; 

	

	  /*...->0*/

	  if((arr[k].description)[4*prva] == '1' &&
		  (arr[k].description)[4*druga] == '1' && 
     (arr[k].description)[4*treca] == '1' && 
	 (arr[k].description)[4*vred] == '0') return 0; 

	   if((arr[k].description)[4*prva+1] == '1' &&
		   (arr[k].description)[4*druga+1] == '1' && 
     (arr[k].description)[4*treca+1] == '1' &&
	 (arr[k].description)[4*vred+1] == '0') return 0; 

	    if((arr[k].description)[4*prva+2] == '1' &&
			(arr[k].description)[4*druga+2] == '1' && 
     (arr[k].description)[4*treca+2] == '1' &&
	 (arr[k].description)[4*vred+2] == '0') return 0; 

		 if((arr[k].description)[4*prva+3] == '1' &&
			 (arr[k].description)[4*druga+3] == '1' && 
     (arr[k].description)[4*treca+3] == '1' && 
	 (arr[k].description)[4*vred+3] == '0') return 0; 

		/*previous members*/

		 for(j=1; j<=i; j++) {

				 pj=((*(p-j)).v)[0]-'0';
				 dj=((*(p-j)).v)[1]-'0';
				 tj=((*(p-j)).v)[2]-'0';
				 vj=(*(p-j)).value -'0';  

				

	 if((arr[k].description)[4*pj+prva] == '1' && 
		 (arr[k].description)[4*dj+druga] == '1' && 
		 (arr[k].description)[4*tj+treca] == '1' && 
		 (arr[k].description)[4*vj+vred] == '0') {rez=0; break;}

	 
    if((arr[k].description)[4*prva+pj] == '1' &&
		(arr[k].description)[4*druga+dj] == '1' && 
	 (arr[k].description)[4*treca+tj] == '1' && 
	 (arr[k].description)[4*vred+vj] == '0') {rez=0; break;}

	/* if j<12 we test all the permutations*/

	if(i-j<12){ 
		
		if((arr[k].description)[4*pj+prva]=='1'&&
			(arr[k].description)[4*tj+druga]=='1'
	&& (arr[k].description)[4*dj+treca]=='1' && 
	(arr[k].description)[4*vj+vred]=='0') { rez =0; break;}

		if( (arr[k].description)[4*tj+prva]=='1' && 
			(arr[k].description)[4*pj+druga]=='1'
	&& (arr[k].description)[4*dj+treca]=='1'
	&& (arr[k].description)[4*vj+vred]=='0') { rez =0; break;}

		if((arr[k].description)[4*prva+pj]=='1'&& 
			(arr[k].description)[4*druga+tj]=='1'
	&& (arr[k].description)[4*treca+dj]=='1' && 
	(arr[k].description)[4*vred+vj]=='0') { rez =0; break;}

		if((arr[k].description)[4*prva+tj]=='1' &&
			(arr[k].description)[4*druga+pj]=='1'
			&& (arr[k].description)[4*treca+dj]=='1' && 
			(arr[k].description)[4*vred+vj]=='0') { rez =0; break;}}

		 } /*for*/

		 if(rez ==0) return 0; 

		 return 1; /*ok*/

		 } /*check2*/

 

		 int g_wnu3(int k){
	  int i;
	
	  for(i=4; i<36; i++) arr2[i].value ='a';
	 
	  g_3(&arr2[4],4,k); 

      if( arr2[35].value =='a') return 0; 

	  return 1; } /* g_wnu3*/


  void g_3(struct m *p, int i, int k){



	  int x=0, y=1; 
	  while (!x && y) { y=0; 
	                    (*p).value ='0';
						x=check3(p, i, k); 

						if(!x) { (*p).value ='1';
						      x=check3(p, i, k); }
						if(!x) { (*p).value ='2';
						      x=check3(p, i, k); }
						if(!x) { (*p).value ='3';
						      x=check3(p, i, k); }
						if(!x) { (*p).value ='a';
						if(i>4) y=back_3(p-1, i-1,k); } } /*while*/

	  if(x==1 && i<35) g_3(p+1,i+1,k); 
	  else return;} /* g_3*/


   

  int back_3(struct m *p, int i, int k){

	  int x=0, y=1;

	  
	  if((*p).value =='3' && i==4) return 0; 

	  while (x==0 && (*p).value<'3' ) {

		  (*p).value+=1;
		  x=check3(p,i,k);} /*while*/
	  if(x==1) return 1;
	  else if (i==4) return 0; 
	  else{ 

		  while( !x && y){ y=0; 
		         (*p).value ='a';
				 y = back_3(p-1, i-1, k);
				 if(y){  (*p).value ='0';
				         x=check3(p,i,k);
						 if(!x) { (*p).value ='1';
						 x=check3(p,i,k);}
						 if(!x) { (*p).value ='2';
						 x=check3(p,i,k);}
						 if(!x) { (*p).value ='3';
						 x=check3(p,i,k);} } 
		  } /*while*/

		  if(x==1) return 1;
		  return 0;} /*else*/
                } /*back_3*/






  int check3(struct m *p, int i, int k){



	  int prva, druga, treca, pj, dj, tj, vred, vj, i1, rez, j;
	  rez =1;
	  prva = ((*p).v)[0] -'0';
	  druga = ((*p).v)[1] -'0';
	  treca = ((*p).v)[2] -'0';  

	  vred = (*p).value -'0';  

	  

      if((i==12 || i==14 || i==18 || i==20 || i==24 || i==26) && vred == 3 &&
		  subalgebras[k][6]=='1') return 0;

	  if((i==13 || i==16 || i==19 || i==22 || i==30 || i==32) && vred == 2 &&
		  subalgebras[k][7]=='1') return 0;

	   if((i==15 || i==17 || i==25 || i==28 || i==31 || i==34) && vred == 1 &&
		  subalgebras[k][8]=='1') return 0;


	   if((i==21 || i==23 || i==27 || i==29 || i==33 || i==35) && vred == 0  &&
		  subalgebras[k][9]=='1') return 0;



	  /* 0->...*/
	  if((arr[k].description)[prva] == '1' && 
           (arr[k].description)[druga] == '1' && 
     (arr[k].description)[treca] == '1' && 
	 (arr[k].description)[vred] == '0') return 0; 

       /* 1->...*/

      if((arr[k].description)[4+prva] == '1' && 
       (arr[k].description)[4+druga] == '1' && 
     (arr[k].description)[4+treca] == '1' &&
	 (arr[k].description)[4+vred] == '0') return 0;

	  if((arr[k].description)[8+prva] == '1' && 
        (arr[k].description)[8+druga] == '1' && 
     (arr[k].description)[8+treca] == '1' && 
	 (arr[k].description)[8+vred] == '0') return 0;

	  if((arr[k].description)[12+prva] == '1' && 
         (arr[k].description)[12+druga] == '1' && 
     (arr[k].description)[12+treca] == '1' &&
	 (arr[k].description)[12+vred] == '0') return 0; 

	  

	  /*...->0*/

	  if((arr[k].description)[4*prva] == '1' && 
		  (arr[k].description)[4*druga] == '1' && 
     (arr[k].description)[4*treca] == '1' && 
	 (arr[k].description)[4*vred] == '0') return 0; 

	   if((arr[k].description)[4*prva+1] == '1' &&
		   (arr[k].description)[4*druga+1] == '1' && 
     (arr[k].description)[4*treca+1] == '1' && 
	 (arr[k].description)[4*vred+1] == '0') return 0; 

	    if((arr[k].description)[4*prva+2] == '1' && 
			(arr[k].description)[4*druga+2] == '1' && 
     (arr[k].description)[4*treca+2] == '1' && 
	 (arr[k].description)[4*vred+2] == '0') return 0; 

		 if((arr[k].description)[4*prva+3] == '1' && 
			 (arr[k].description)[4*druga+3] == '1' && 
     (arr[k].description)[4*treca+3] == '1' && 
	 (arr[k].description)[4*vred+3] == '0') return 0; 

		

		 for(j=1; j<=i; j++) {

				 pj=((*(p-j)).v)[0]-'0';
				 dj=((*(p-j)).v)[1]-'0';
				 tj=((*(p-j)).v)[2]-'0';
				 vj=(*(p-j)).value -'0';  

				 

	 if((arr[k].description)[4*pj+prva] == '1' &&
		 (arr[k].description)[4*dj+druga] == '1' && 
		 (arr[k].description)[4*tj+treca] == '1' &&
		 (arr[k].description)[4*vj+vred] == '0') 
	 {rez=0; break;}

	
    if((arr[k].description)[4*prva+pj] == '1' &&
		(arr[k].description)[4*druga+dj] == '1' && 
	 (arr[k].description)[4*treca+tj] == '1' && 
	 (arr[k].description)[4*vred+vj] == '0') 
	{rez=0; break;}

	

	if(i-j<12 && i-j>=0 ){ 
		
		if((arr[k].description)[4*pj+prva]=='1'&& 
			(arr[k].description)[4*tj+druga]=='1'
	&& (arr[k].description)[4*dj+treca]=='1' && 
	(arr[k].description)[4*vj+vred]=='0') 
		{ rez =0; break;}

		if( (arr[k].description)[4*tj+prva]=='1' && 
			(arr[k].description)[4*pj+druga]=='1'
	&& (arr[k].description)[4*dj+treca]=='1'
	&& (arr[k].description)[4*vj+vred]=='0') { rez =0; break;}

		if((arr[k].description)[4*prva+pj]=='1'&&
			(arr[k].description)[4*druga+tj]=='1'
	&& (arr[k].description)[4*treca+dj]=='1' &&
	(arr[k].description)[4*vred+vj]=='0')
		{ rez =0; break;}

		if((arr[k].description)[4*prva+tj]=='1' &&
			(arr[k].description)[4*druga+pj]=='1'
			&& (arr[k].description)[4*treca+dj]=='1' &&
			(arr[k].description)[4*vred+vj]=='0') 
		{ rez =0; break;}}

		 } /*for*/

		 if(rez ==0) return 0; 

		 return 1; /*ok*/

		 } /*check3*/





	void set_arr5(void){

		int i0,i1,i2,i3,i4,i5,i6,i7,i8,i9,i10,i11, j; 

        j=0;
		for(i0=0; i0<2; i0++)
			for(i1=0; i1<3; i1+=2)
				for(i2=0; i2<4; i2+=3)
                 for(i3=0; i3<2; i3++)
					 for(i4=1; i4<3; i4++)
						 for(i5=1; i5<4; i5+=2)
                            for(i6=0; i6<3; i6+=2)
								for(i7=1; i7<3; i7++)
									for(i8=2; i8<4; i8++)
										for(i9=0; i9<4; i9+=3)
											for(i10=1; i10<4; i10+=2)
												for(i11=2; i11<4; i11++)
												{arr5[j][0]=i0+'0';
		                                         arr5[j][1]=i1+'0';
												 arr5[j][2]=i2+'0';
												 arr5[j][3]=i3+'0';
												 arr5[j][4]=i4+'0';
												 arr5[j][5]=i5+'0';
												 arr5[j][6]=i6+'0';
												 arr5[j][7]=i7+'0';
												 arr5[j][8]=i8+'0';
												 arr5[j][9]=i9+'0';
												 arr5[j][10]=i10+'0';
												arr5[j][11]=i11+'0';
												j++;}
	}/*set_arr5*/




	

	int check5( char *p, int i){  int j,k; int pj,dj,tj,vj, pk,dk,tk,vk;  int rez=1; 

		

		(arr3[0]).v = "001"; 
		(arr3[1]).v = "002"; 
		(arr3[2]).v = "003"; 
		(arr3[3]).v = "110";
		(arr3[4]).v = "112"; 
		(arr3[5]).v = "113";
		(arr3[6]).v = "220"; 
		(arr3[7]).v = "221";
		(arr3[8]).v = "223"; 
		(arr3[9]).v = "330"; 
		(arr3[10]).v = "331";
		(arr3[11]).v = "332"; 

		for( j=0; j<12; j++)  arr3[j].value= p[j]; 


		

		
		for( j=0; j<12; j++) { 

			pj=(arr3[j].v)[0] -'0'; 
			dj=(arr3[j].v)[1] -'0';
			tj=(arr3[j].v)[2] -'0'; 
			vj=(arr3[j].value) -'0';  

			for( k=j-1; k>=0; k--) {

            pk=(arr3[k].v)[0] -'0'; 
			dk=(arr3[k].v)[1] -'0';
			tk=(arr3[k].v)[2] -'0'; 
			vk=(arr3[k].value) -'0';  


			

			if((((arr[i].description)[4*pj+pk] == '1' && 
				(arr[i].description)[4*dj+dk] == '1' 
				&& (arr[i].description)[4*tj+tk] == '1')||
				((arr[i].description)[4*pj+pk] == '1' &&
				(arr[i].description)[4*dj+tk] == '1' 
				&& (arr[i].description)[4*tj+dk] == '1') ||
				((arr[i].description)[4*pj+tk] == '1' && 
				(arr[i].description)[4*dj+pk] == '1' 
				&& (arr[i].description)[4*tj+dk] == '1')) &&
				(arr[i].description)[4*vj+vk]=='0') 
			{rez=0; break;}

			

            if((((arr[i].description)[4*pk+pj] == '1' &&
				(arr[i].description)[4*dk+dj] == '1' 
				&& (arr[i].description)[4*tk+tj] == '1')||
				((arr[i].description)[4*pk+pj] == '1' &&
				(arr[i].description)[4*tk+dj] == '1' 
				&& (arr[i].description)[4*dk+tj] == '1') ||
				((arr[i].description)[4*tk+pj] == '1' &&
				(arr[i].description)[4*pk+dj] == '1' 
				&& (arr[i].description)[4*dk+tj] == '1')) &&
				(arr[i].description)[4*vk+vj]=='0')
			{rez=0; break;}
			}
		    if(rez ==0) break;}


		if(rez ==0) return 0; 

		return 1; /*ok*/

	} /* check5*/

	

	void set12_arr2(int j) { int k; 

			(arr2[0]).v = "001"; 
			(arr2[1]).v = "002"; 
			(arr2[2]).v = "003"; 
			(arr2[3]).v = "110";
			(arr2[4]).v = "112"; 
			(arr2[5]).v = "113";
			(arr2[6]).v = "220"; 
			(arr2[7]).v = "221";
			(arr2[8]).v = "223"; 
			(arr2[9]).v = "330"; 
			(arr2[10]).v = "331";
			(arr2[11]).v = "332"; 
			arr2[12].v = "012";
	        arr2[13].v = "013";
	        arr2[14].v = "021";
	        arr2[15].v = "023";
	        arr2[16].v = "031";
	        arr2[17].v = "032";
	        arr2[18].v = "102";
	        arr2[19].v = "103";
	        arr2[20].v = "120";
	        arr2[21].v = "123";
	        arr2[22].v = "130";
	        arr2[23].v = "132";
	        arr2[24].v = "201";
	        arr2[25].v = "203";
	        arr2[26].v = "210";
	        arr2[27].v = "213";
	        arr2[28].v = "230";
	        arr2[29].v = "231";
	        arr2[30].v = "301";
	        arr2[31].v = "302";
	        arr2[32].v = "310";
	        arr2[33].v = "312";
	        arr2[34].v = "320";
	        arr2[35].v = "321";

		for(k=0; k<12; k++) arr2[k].value = arr5[j][k];

		for(k=12; k<36; k++) arr2[k].value = 'a';} /* set12_arr2*/



	void n_arr2(void) { int k; 

	   (arr2[0]).v = "001"; 
	   (arr2[1]).v = "002"; 
	   (arr2[2]).v = "003"; 
	   (arr2[3]).v = "110";
	   (arr2[4]).v = "112"; 
	   (arr2[5]).v = "113";
	   (arr2[6]).v = "220"; 
	   (arr2[7]).v = "221";
	   (arr2[8]).v = "223"; 
	   (arr2[9]).v = "330"; 
	   (arr2[10]).v = "331";
	   (arr2[11]).v = "332"; 
	    arr2[12].v = "012";
	    arr2[13].v = "013";
	    arr2[14].v = "021";
	    arr2[15].v = "023";
	    arr2[16].v = "031";
	    arr2[17].v = "032";
	    arr2[18].v = "102";
	    arr2[19].v = "103";
	    arr2[20].v = "120";
	    arr2[21].v = "123";
	    arr2[22].v = "130";
	    arr2[23].v = "132";
	    arr2[24].v = "201";
	    arr2[25].v = "203";
	    arr2[26].v = "210";
	    arr2[27].v = "213";
	    arr2[28].v = "230";
        arr2[29].v = "231";
	    arr2[30].v = "301";
	    arr2[31].v = "302";
	    arr2[32].v = "310";
	    arr2[33].v = "312";
	    arr2[34].v = "320";
	    arr2[35].v = "321";

	    arr2[0].value ='0';
	    arr2[1].value ='0';
	    arr2[2].value ='0';
	    arr2[3].value ='1';
	    arr2[4].value ='1';
	    arr2[5].value ='1';
	    arr2[6].value ='2';
	    arr2[7].value ='2';
	    arr2[8].value ='2';
	    arr2[9].value ='3';
	    arr2[10].value ='3';
	    arr2[11].value ='3';

		for(k=12; k<36; k++) arr2[k].value = 'a';} /* n_arr2*/


	

	void set_arr6(void){

		int i0,i1,i2,i3,i4,i5,i6,i7,i8,i9, j; 

        j=0;
		for(i0=0; i0<2; i0++)
			for(i1=0; i1<3; i1+=2)
				for(i2=0; i2<4; i2+=3)
                 for(i3=0; i3<2; i3++)
					 for(i4=1; i4<3; i4++)
						 for(i5=1; i5<4; i5+=2)
                            for(i6=0; i6<3; i6+=2)
								for(i7=1; i7<3; i7++)
									for(i8=0; i8<4; i8+=3)
										for(i9=1; i9<4; i9+=2)
											
												{arr6[j][0]=i0+'0';
		                                         arr6[j][1]=i1+'0';
												 arr6[j][2]=i2+'0';
												 arr6[j][3]=i3+'0';
												 arr6[j][4]=i4+'0';
												 arr6[j][5]=i5+'0';
												 arr6[j][6]=i6+'0';
												 arr6[j][7]=i7+'0';
												 arr6[j][8]=i8+'0';
												 arr6[j][9]=i9+'0';
												 j++;}
	}/*set_arr6*/


	

	int check6( char *p, int i){  int j,k; int pj,dj,tj,vj, pk,dk,tk,vk;  int rez=1; 

		

		(arr7[0]).v = "001"; 
		(arr7[1]).v = "002"; 
		(arr7[2]).v = "003"; 
		(arr7[3]).v = "110";
		(arr7[4]).v = "112"; 
		(arr7[5]).v = "113";
		(arr7[6]).v = "220"; 
		(arr7[7]).v = "221";
		(arr7[8]).v = "330"; 
		(arr7[9]).v = "331"; 
		

		for( j=0; j<10; j++)  arr7[j].value= p[j]; 


		

		
		for( j=0; j<10; j++) { 

			pj=(arr7[j].v)[0] -'0'; 
			dj=(arr7[j].v)[1] -'0';
			tj=(arr7[j].v)[2] -'0'; 
			vj=(arr7[j].value) -'0';  

			for( k=j-1; k>=0; k--) {

            pk=(arr7[k].v)[0] -'0'; 
			dk=(arr7[k].v)[1] -'0';
			tk=(arr7[k].v)[2] -'0'; 
			vk=(arr7[k].value) -'0';   

			

			if((((arr[i].description)[4*pj+pk] == '1' && 
				(arr[i].description)[4*dj+dk] == '1' 
				&& (arr[i].description)[4*tj+tk] == '1')||
				((arr[i].description)[4*pj+pk] == '1' && 
				(arr[i].description)[4*dj+tk] == '1' 
				&& (arr[i].description)[4*tj+dk] == '1') ||
				((arr[i].description)[4*pj+tk] == '1' && 
                                (arr[i].description)[4*dj+pk] == '1' 
				&& (arr[i].description)[4*tj+dk] == '1')) 
				&& (arr[i].description)[4*vj+vk]=='0')
			{rez=0; break;}

			

            if((((arr[i].description)[4*pk+pj] == '1' && 
				(arr[i].description)[4*dk+dj] == '1' 
				&& (arr[i].description)[4*tk+tj] == '1')||
				((arr[i].description)[4*pk+pj] == '1' && 
				(arr[i].description)[4*tk+dj] == '1' 
				&& (arr[i].description)[4*dk+tj] == '1') ||
				((arr[i].description)[4*tk+pj] == '1' && 
				(arr[i].description)[4*pk+dj] == '1' 
				&& (arr[i].description)[4*dk+tj] == '1')) &&
				(arr[i].description)[4*vk+vj]=='0')
			{rez=0; break;}
			}
		    if(rez ==0) break;}


		if(rez ==0) return 0; 

		return 1; /*ok*/

	} /* check6*/


	void n6_arr2(int j) { int k; 

		(arr2[0]).v = "001"; 
		(arr2[1]).v = "002"; 
		(arr2[2]).v = "003"; 
		(arr2[3]).v = "110";
		(arr2[4]).v = "112"; 
		(arr2[5]).v = "113";
		(arr2[6]).v = "220"; 
		(arr2[7]).v = "221";
		(arr2[8]).v = "330"; 
		(arr2[9]).v = "331"; 
        (arr2[10]).v = "223";
		(arr2[11]).v = "332"; 
		arr2[12].v = "012";
	    arr2[13].v = "013";
	    arr2[14].v = "021";
	    arr2[15].v = "023";
	    arr2[16].v = "031";
	    arr2[17].v = "032";
	    arr2[18].v = "102";
	    arr2[19].v = "103";
	    arr2[20].v = "120";
	    arr2[21].v = "123";
	    arr2[22].v = "130";
	    arr2[23].v = "132";
	    arr2[24].v = "201";
	    arr2[25].v = "203";
	    arr2[26].v = "210";
	    arr2[27].v = "213";
	    arr2[28].v = "230";
	    arr2[29].v = "231";
	    arr2[30].v = "301";
	    arr2[31].v = "302";
	    arr2[32].v = "310";
	    arr2[33].v = "312";
	    arr2[34].v = "320";
	    arr2[35].v = "321";

		for(k=0; k<10; k++) arr2[k].value = arr6[j][k];

		for(k=10; k<36; k++) arr2[k].value = 'a';} /* n6_arr2*/


	

	void set_arr8(void){

		int i0,i1,i2,i3,i4,i5,i6,i7, j; 

        j=0;
		for(i0=0; i0<3; i0+=2) /*02*/
			for(i1=0; i1<4; i1+=3)/*03*/
				for(i2=1; i2<3; i2++)/*12*/
                 for(i3=0; i3<3; i3+=2)/*02*/
					 for(i4=1; i4<3; i4++)/*21*/
						 for(i5=2; i5<4; i5++)/*23*/
                            for(i6=0; i6<4; i6+=3)/*03*/
								for(i7=2; i7<4; i7++)/*23*/
									           {arr8[j][0]=i0+'0';
		                                         arr8[j][1]=i1+'0';
												 arr8[j][2]=i2+'0';
												 arr8[j][3]=i3+'0';
												 arr8[j][4]=i4+'0';
												 arr8[j][5]=i5+'0';
												 arr8[j][6]=i6+'0';
												 arr8[j][7]=i7+'0';
												  j++;}
	}/*set_arr8*/

	int check8( char *p, int i){  int j,k; int pj,dj,tj,vj, pk,dk,tk,vk;  int rez=1; 

		

		(arr9[0]).v = "002"; 
		(arr9[1]).v = "003"; 
		(arr9[2]).v = "112"; 
		(arr9[3]).v = "220";
		(arr9[4]).v = "221"; 
		(arr9[5]).v = "223";
		(arr9[6]).v = "330"; 
		(arr9[7]).v = "332";
		
		

		for( j=0; j<8; j++)  arr9[j].value= p[j]; 


		

		
		for( j=0; j<8; j++) { 

			pj=(arr9[j].v)[0] -'0'; 
			dj=(arr9[j].v)[1] -'0';
			tj=(arr9[j].v)[2] -'0'; 
			vj=(arr9[j].value) -'0';  

			for( k=j-1; k>=0; k--) {

            pk=(arr9[k].v)[0] -'0'; 
			dk=(arr9[k].v)[1] -'0';
			tk=(arr9[k].v)[2] -'0'; 
			vk=(arr9[k].value) -'0';   


			

			if((((arr[i].description)[4*pj+pk] == '1' &&
				(arr[i].description)[4*dj+dk] == '1' 
				&& (arr[i].description)[4*tj+tk] == '1')||
				((arr[i].description)[4*pj+pk] == '1' &&
				(arr[i].description)[4*dj+tk] == '1' 
				&& (arr[i].description)[4*tj+dk] == '1') ||
				((arr[i].description)[4*pj+tk] == '1' && 
				(arr[i].description)[4*dj+pk] == '1' 
				&& (arr[i].description)[4*tj+dk] == '1')) 
				&& (arr[i].description)[4*vj+vk]=='0')
			{rez=0; break;}

			

            if((((arr[i].description)[4*pk+pj] == '1' &&
				(arr[i].description)[4*dk+dj] == '1' 
				&& (arr[i].description)[4*tk+tj] == '1')||
				((arr[i].description)[4*pk+pj] == '1' && 
				(arr[i].description)[4*tk+dj] == '1' 
				&& (arr[i].description)[4*dk+tj] == '1') ||
				((arr[i].description)[4*tk+pj] == '1' &&
				(arr[i].description)[4*pk+dj] == '1' 
				&& (arr[i].description)[4*dk+tj] == '1')) && 
				(arr[i].description)[4*vk+vj]=='0')
			{rez=0; break;}
			}
		    if(rez ==0) break;}


		if(rez ==0) return 0; 

		return 1; /*ok*/

	} /* check8*/

	void n8_arr2(int j) { int k; 

	   (arr2[0]).v = "002"; 
	   (arr2[1]).v = "003"; 
	   (arr2[2]).v = "112"; 
	   (arr2[3]).v = "220";
	   (arr2[4]).v = "221"; 
	   (arr2[5]).v = "223";
	   (arr2[6]).v = "330"; 
	   (arr2[7]).v = "332";
	   (arr2[8]).v = "001"; 
       (arr2[9]).v = "110"; 
       (arr2[10]).v = "113";
	   (arr2[11]).v = "331"; 
	    arr2[12].v = "012";
	    arr2[13].v = "013";
	    arr2[14].v = "021";
	    arr2[15].v = "023";
	    arr2[16].v = "031";
	    arr2[17].v = "032";
	    arr2[18].v = "102";
	    arr2[19].v = "103";
	    arr2[20].v = "120";
	    arr2[21].v = "123";
	    arr2[22].v = "130";
	    arr2[23].v = "132";
	    arr2[24].v = "201";
	    arr2[25].v = "203";
	    arr2[26].v = "210";
	    arr2[27].v = "213";
	    arr2[28].v = "230";
	    arr2[29].v = "231";
	    arr2[30].v = "301";
	    arr2[31].v = "302";
	    arr2[32].v = "310";
	    arr2[33].v = "312";
	    arr2[34].v = "320";
	    arr2[35].v = "321";

		for(k=0; k<8; k++) arr2[k].value = arr8[j][k];

		for(k=8; k<36; k++) arr2[k].value = 'a';} /* n8_arr2*/

	void set_arr10(void){

		int i0,i1,i2,i3,i4,i5, j; 

        j=0;
		for(i0=1; i0<3; i0++)/*12*/
			for(i1=1; i1<4; i1+=2)/*13*/
				for(i2=1; i2<3; i2++)/*12*/
                 for(i3=2; i3<4; i3++)/*23*/
					 for(i4=1; i4<4; i4+=2)/*13*/
						 for(i5=2; i5<4; i5++)/*23*/
                           
									           {arr10[j][0]=i0+'0';
		                                         arr10[j][1]=i1+'0';
												 arr10[j][2]=i2+'0';
												 arr10[j][3]=i3+'0';
												 arr10[j][4]=i4+'0';
												 arr10[j][5]=i5+'0';
												 
												  j++;}
	}/*set_arr10*/

	int check10( char *p, int i){  int j,k; int pj,dj,tj,vj, pk,dk,tk,vk;  int rez=1; 

		

		(arr11[0]).v = "112"; 
		(arr11[1]).v = "113"; 
		(arr11[2]).v = "221"; 
		(arr11[3]).v = "223";
		(arr11[4]).v = "331"; 
		(arr11[5]).v = "332";
		
		

		for( j=0; j<6; j++)  arr11[j].value= p[j]; 


		

		
		for( j=0; j<6; j++) { 

			pj=(arr11[j].v)[0] -'0'; 
			dj=(arr11[j].v)[1] -'0';
			tj=(arr11[j].v)[2] -'0'; 
			vj=(arr11[j].value) -'0'; 

			for( k=j-1; k>=0; k--) { 

            pk=(arr11[k].v)[0] -'0'; 
			dk=(arr11[k].v)[1] -'0';
			tk=(arr11[k].v)[2] -'0'; 
			vk=(arr11[k].value) -'0'; 


			

			if((((arr[i].description)[4*pj+pk] == '1' &&
				(arr[i].description)[4*dj+dk] == '1' 
				&& (arr[i].description)[4*tj+tk] == '1')||
				((arr[i].description)[4*pj+pk] == '1' && 
				(arr[i].description)[4*dj+tk] == '1' 
				&& (arr[i].description)[4*tj+dk] == '1') ||
				((arr[i].description)[4*pj+tk] == '1' && 
                                (arr[i].description)[4*dj+pk] == '1' 
				&& (arr[i].description)[4*tj+dk] == '1')) &&
				(arr[i].description)[4*vj+vk]=='0') 
			{rez=0; break;}

			

            if((((arr[i].description)[4*pk+pj] == '1' && 
				(arr[i].description)[4*dk+dj] == '1' 
				&& (arr[i].description)[4*tk+tj] == '1')||
				((arr[i].description)[4*pk+pj] == '1' &&
				(arr[i].description)[4*tk+dj] == '1' 
				&& (arr[i].description)[4*dk+tj] == '1') ||
				((arr[i].description)[4*tk+pj] == '1' &&
				(arr[i].description)[4*pk+dj] == '1' 
				&& (arr[i].description)[4*dk+tj] == '1')) &&
				(arr[i].description)[4*vk+vj]=='0')
			{rez=0; break;}
			}
		    if(rez ==0) break;}


		if(rez ==0) return 0; 

		return 1; /*ok*/

	} /* check10*/

	void n10_arr2(int j) { int k; 

			(arr2[0]).v = "112"; 
			(arr2[1]).v = "113"; 
			(arr2[2]).v = "221"; 
			(arr2[3]).v = "223";
			(arr2[4]).v = "331"; 
			(arr2[5]).v = "332";
			(arr2[6]).v = "001"; 
			(arr2[7]).v = "110";
			(arr2[8]).v = "002"; 
			(arr2[9]).v = "220"; 
            (arr2[10]).v = "003";
			(arr2[11]).v = "330"; 
			arr2[12].v = "012";
	        arr2[13].v = "013";
	        arr2[14].v = "021";
	    	arr2[15].v = "023";
	    	arr2[16].v = "031";
	    	arr2[17].v = "032";
	    	arr2[18].v = "102";
	    	arr2[19].v = "103";
	    	arr2[20].v = "120";
	    	arr2[21].v = "123";
	    	arr2[22].v = "130";
	    	arr2[23].v = "132";
	    	arr2[24].v = "201";
	    	arr2[25].v = "203";
	    	arr2[26].v = "210";
	   		arr2[27].v = "213";
	    	arr2[28].v = "230";
			arr2[29].v = "231";
	    	arr2[30].v = "301";
	    	arr2[31].v = "302";
	    	arr2[32].v = "310";
	    	arr2[33].v = "312";
	    	arr2[34].v = "320";
	    	arr2[35].v = "321";

		for(k=0; k<6; k++) arr2[k].value = arr10[j][k];

		for(k=6; k<36; k++) arr2[k].value = 'a';} /* n10_arr2*/


	void set_arr12(void){

		int i0,i1,i2,i3, j; 

        j=0;
		for(i0=0; i0<2; i0++)/*01*/
			for(i1=0; i1<3; i1+=2)/*02*/
				for(i2=0; i2<2; i2++)/*01*/
                 for(i3=0; i3<3; i3+=2)/*02*/
			                                    {arr12[j][0]=i0+'0';
		                                         arr12[j][1]=i1+'0';
												 arr12[j][2]=i2+'0';
												 arr12[j][3]=i3+'0';
												 j++;}
	}/*set_arr12*/


	int check12( char *p, int i){  int j,k; int pj,dj,tj,vj, pk,dk,tk,vk;  int rez=1; 

		

		(arr13[0]).v = "001"; 
		(arr13[1]).v = "002"; 
		(arr13[2]).v = "110"; 
		(arr13[3]).v = "220";
		
		for( j=0; j<4; j++)  arr13[j].value= p[j]; 


		

		
		for( j=0; j<4; j++) { 

			pj=(arr13[j].v)[0] -'0'; 
			dj=(arr13[j].v)[1] -'0';
			tj=(arr13[j].v)[2] -'0'; 
			vj=(arr13[j].value) -'0';  

			for( k=j-1; k>=0; k--) {

            pk=(arr13[k].v)[0] -'0'; 
			dk=(arr13[k].v)[1] -'0';
			tk=(arr13[k].v)[2] -'0'; 
			vk=(arr13[k].value) -'0';   


		

			if((((arr[i].description)[4*pj+pk] == '1' && 
				(arr[i].description)[4*dj+dk] == '1' 
				&& (arr[i].description)[4*tj+tk] == '1')||
				((arr[i].description)[4*pj+pk] == '1' && 
				(arr[i].description)[4*dj+tk] == '1' 
				&& (arr[i].description)[4*tj+dk] == '1') ||
				((arr[i].description)[4*pj+tk] == '1' && 
                                 (arr[i].description)[4*dj+pk] == '1' 
				&& (arr[i].description)[4*tj+dk] == '1'))
				&& (arr[i].description)[4*vj+vk]=='0')
			{rez=0; break;}

			

            if((((arr[i].description)[4*pk+pj] == '1' &&
				(arr[i].description)[4*dk+dj] == '1' 
				&& (arr[i].description)[4*tk+tj] == '1')||
				((arr[i].description)[4*pk+pj] == '1' && 
				(arr[i].description)[4*tk+dj] == '1' 
				&& (arr[i].description)[4*dk+tj] == '1') ||
				((arr[i].description)[4*tk+pj] == '1' && 
                                 (arr[i].description)[4*pk+dj] == '1' 
				&& (arr[i].description)[4*dk+tj] == '1'))
				&& (arr[i].description)[4*vk+vj]=='0')
			{rez=0; break;}
			}
		    if(rez ==0) break;}


		if(rez ==0) return 0; 

		return 1; /*ok*/

	} /* check12*/

	void n12_arr2(int j) { int k; 

			(arr2[0]).v = "001"; 
			(arr2[1]).v = "002"; 
			(arr2[2]).v = "110"; 
			(arr2[3]).v = "220";
			(arr2[4]).v = "003"; 
			(arr2[5]).v = "330";
			(arr2[6]).v = "112"; 
			(arr2[7]).v = "221";
			(arr2[8]).v = "113"; 
			(arr2[9]).v = "331"; 
            (arr2[10]).v = "223";
		    (arr2[11]).v = "332"; 
		    arr2[12].v = "012";
	    	arr2[13].v = "013";
	    	arr2[14].v = "021";
	    	arr2[15].v = "023";
	    	arr2[16].v = "031";
	    	arr2[17].v = "032";
	     	arr2[18].v = "102";
	    	arr2[19].v = "103";
	    	arr2[20].v = "120";
	    	arr2[21].v = "123";
	    	arr2[22].v = "130";
	    	arr2[23].v = "132";
	    	arr2[24].v = "201";
	    	arr2[25].v = "203";
	    	arr2[26].v = "210";
	    	arr2[27].v = "213";
	   	    arr2[28].v = "230";
		    arr2[29].v = "231";
	    	arr2[30].v = "301";
	    	arr2[31].v = "302";
	    	arr2[32].v = "310";
	    	arr2[33].v = "312";
	    	arr2[34].v = "320";
	    	arr2[35].v = "321";

		for(k=0; k<4; k++) arr2[k].value = arr12[j][k];

		for(k=4; k<36; k++) arr2[k].value = 'a';} /* n12_arr2*/
		

void set_newarray()
{ 
     int i; 

	 newarray[0].v="010";
	 newarray[1].v="020"; 
	 newarray[2].v="030"; 
	 newarray[3].v="101"; 
	 newarray[4].v="121"; 
	 newarray[5].v="131"; 
	 newarray[6].v="202"; 
	 newarray[7].v="212"; 
	 newarray[8].v="232"; 
	 newarray[9].v="303";
	 newarray[10].v="313"; 
	 newarray[11].v="323"; 
	 newarray[12].v="001";
	 newarray[13].v="002";
	 newarray[14].v="003"; 
	 newarray[15].v="110"; 
	 newarray[16].v="112";
	 newarray[17].v="113";
	 newarray[18].v="220"; 
	 newarray[19].v="221"; 
	 newarray[20].v="223";
	 newarray[21].v="330"; 
	 newarray[22].v="331";
	 newarray[23].v="332";
	 newarray[24].v="012"; 
	 newarray[25].v="013";
	 newarray[26].v="021"; 
	 newarray[27].v="023";
	 newarray[28].v="031";
	 newarray[29].v="032";
	 newarray[30].v="102";
	 newarray[31].v="103";
	 newarray[32].v="120";
	 newarray[33].v="123";
	 newarray[34].v="130"; 
	 newarray[35].v="132"; 
	 newarray[36].v="201";
	 newarray[37].v="203"; 
	 newarray[38].v="210"; 
	 newarray[39].v="213"; 
	 newarray[40].v="230";
	 newarray[41].v="231"; 
	 newarray[42].v="301"; 
	 newarray[43].v="302";
	 newarray[44].v="310";
	 newarray[45].v="312";
	 newarray[46].v="320";
	 newarray[47].v="321"; 

	for(i=0; i<12; i++)  newarray[i].value = niz2[i].value;

	for(i=12; i<48; i++) newarray[i].value ='a'; 
 }   /* set_newarray*/
		
		
int my_terms(int k){ 

	   
	  int i;
	
	  for(i=12; i<48; i++) newarray[i].value ='a';
	 
	  m_t(&newarray[12],12,k); 
	  
/* this function will fill the array array3 to the end by the values of the wnu3 term*/

      if( newarray[47].value =='a') return 0; 

	  return 1; }	
	  
void m_t(struct m *p, int i, int k)
{

/* here p is a pointer to the element of the array array3 being currently set,
i is the index of this member of the array,
k is the index of the digraph being examined */

	  int x=0, y=1; 
	  while (!x && y) { y=0; 
	                    (*p).value ='0';
						x=newcheck(p, i, k); 

						if(!x) { (*p).value ='1';
						      x=newcheck(p, i, k); }
						if(!x) { (*p).value ='2';
						      x=newcheck(p, i, k); }
						if(!x) { (*p).value ='3';
						      x=newcheck(p, i, k); }
						if(!x) { (*p).value ='a';
						if(i>12) y=new_back(p-1, i-1,k); } } /*while*/

	  if(x==1 && i<47) m_t(p+1,i+1,k); /* set the next one*/
	  else return;} /* m_t*/
	  
int new_back(struct m *p, int i, int k){

	  int x=0, y=1;

	  /* p is a pointer to the member of the array arr2 whose value is to be reset, 
	  i is the index of this member, k is the index of the digraph being examined */

	  if((*p).value =='3' && i==12) return 0; /* we cannot go any further */

	  while (x==0 && (*p).value<'3' ) {

		  (*p).value+=1;
		  x=newcheck(p,i,k);} /*while*/
	  if(x==1) return 1;
	  else if (i==12) return 0; 
	  else{ 
/* if we are here it means there are previous members of the array arr2, 
so we have to change the previous value */

		  while( !x && y){ y=0; 
		         (*p).value ='a';
				 y = new_back(p-1, i-1, k);
				 if(y){  (*p).value ='0';
				         x=newcheck(p,i,k);
						 if(!x) { (*p).value ='1';
						 x=newcheck(p,i,k);}
						 if(!x) { (*p).value ='2';
						 x=newcheck(p,i,k);}
						 if(!x) { (*p).value ='3';
						 x=newcheck(p,i,k);} } 
		  } /*while*/

		  if(x==1) return 1;
		  return 0;} /*else*/
                } /*new_back*/
                
   int newcheck(struct m *p, int i, int k)
   {

/* p is a pointer to the current ( the last one being set)member of the array arr2, 
i is the index of this member, k is the index of the digraph being examined */

	  int first, second, third, pj, dj, tj, val, vj, i1, res, j;
	  res =1;
	  first = ((*p).v)[0] -'0';
	  second = ((*p).v)[1] -'0';
	  third = ((*p).v)[2] -'0';  /* these are separate digits from the triple being examined */

	  val = (*p).value -'0';  /* this is the assigned value */

	  /* first we deal with subalgebras */

	/*newarray[0].v="010";
	 newarray[1].v="020"; 
	 newarray[2].v="030"; 
	 newarray[3].v="101"; 
	 newarray[4].v="121"; 
	 newarray[5].v="131"; 
	 newarray[6].v="202"; 
	 newarray[7].v="212"; 
	 newarray[8].v="232"; 
	 newarray[9].v="303";
	 newarray[10].v="313"; 
	 newarray[11].v="323"; 
	 newarray[12].v="001";
	 newarray[13].v="002";
	 newarray[14].v="003"; 
	 newarray[15].v="110"; 
	 newarray[16].v="112";
	 newarray[17].v="113";
	 newarray[18].v="220"; 
	 newarray[19].v="221"; 
	 newarray[20].v="223";
	 newarray[21].v="330"; 
	 newarray[22].v="331";
	 newarray[23].v="332";
	 newarray[24].v="012"; 
	 newarray[25].v="013";
	 newarray[26].v="021"; 
	 newarray[27].v="023";
	 newarray[28].v="031";
	 newarray[29].v="032";
	 newarray[30].v="102";
	 newarray[31].v="103";
	 newarray[32].v="120";
	 newarray[33].v="123";
	 newarray[34].v="130"; 
	 newarray[35].v="132"; 
	 newarray[36].v="201";
	 newarray[37].v="203"; 
	 newarray[38].v="210"; 
	 newarray[39].v="213"; 
	 newarray[40].v="230";
	 newarray[41].v="231"; 
	 newarray[42].v="301"; 
	 newarray[43].v="302";
	 newarray[44].v="310";
	 newarray[45].v="312";
	 newarray[46].v="320";
	 newarray[47].v="321"; */


     if((i==0 || i==3 || i==12 ||i==15 ) && (val == 2 || val == 3) && 
     subalgebras[k][0]=='1') return 0;
     if((i==1 || i==6 || i==13 ||i==18 ) && (val == 1 || val == 3) && 
     subalgebras[k][1]=='1') return 0;
	 if((i==2 || i==9 || i==14 ||i==21 ) && (val == 2 || val == 1) && 
	 subalgebras[k][2]=='1') return 0;
	 if((i==4 || i==7 || i==16 ||i==19 ) && (val == 0 || val == 3) && 
	 subalgebras[k][3]=='1') return 0;
	 if((i==5 || i==10 || i==17 ||i==22 ) && (val == 2 || val == 0) && 
	 subalgebras[k][4]=='1') return 0;
	 if((i==8 || i==11 || i==20 ||i==23 ) && (val == 0 || val == 1) && 
	 subalgebras[k][5]=='1') return 0;

      if((i==24 || i==26 || i==30 || i==32 || i==36 || i==38) && val == 3 &&
		  subalgebras[k][6]=='1') return 0;

	  if((i==25 || i==28 || i==31 || i==34 || i==42 || i==44) && val == 2 &&
		  subalgebras[k][7]=='1') return 0;

	   if((i==27 || i==29 || i==37 || i==40 || i==43 || i==46) && val == 1 &&
		  subalgebras[k][8]=='1') return 0;


	   if((i==33 || i==35 || i==39 || i==41 || i==45 || i==47) && val == 0  &&
		  subalgebras[k][9]=='1') return 0;


	  

	  /* we are now finished with subalgebras*/

	  /* examining the relation */

	  /* 0->...*/
	  if((arr[k].description)[first] == '1' && (arr[k].description)[second] == '1' && 
     (arr[k].description)[third] == '1' && (arr[k].description)[val] == '0') 
     return 0; 

       /* 1->...*/

      if((arr[k].description)[4+first] == '1' && (arr[k].description)[4+second] == '1' && 
     (arr[k].description)[4+third] == '1' && (arr[k].description)[4+val] == '0') 
     return 0;

	  if((arr[k].description)[8+first] == '1' && (arr[k].description)[8+second] == '1' && 
     (arr[k].description)[8+third] == '1' && (arr[k].description)[8+val] == '0') 
     return 0;

	  if((arr[k].description)[12+first] == '1' && (arr[k].description)[12+second] == '1' && 
     (arr[k].description)[12+third] == '1' && (arr[k].description)[12+val] == '0') 
     return 0; 

	  /*other way round*/

	  /*...->0*/

	  if((arr[k].description)[4*first] == '1' && (arr[k].description)[4*second] == '1' && 
     (arr[k].description)[4*third] == '1' && (arr[k].description)[4*val] == '0') 
     return 0; 

	   if((arr[k].description)[4*first+1] == '1' && (arr[k].description)[4*second+1] == '1' && 
     (arr[k].description)[4*third+1] == '1' && (arr[k].description)[4*val+1] == '0') 
     return 0; 

	    if((arr[k].description)[4*first+2] == '1' && (arr[k].description)[4*second+2] == '1' && 
     (arr[k].description)[4*third+2] == '1' && (arr[k].description)[4*val+2] == '0') 
     return 0; 

		 if((arr[k].description)[4*first+3] == '1' && (arr[k].description)[4*second+3] == '1' && 
     (arr[k].description)[4*third+3] == '1' && (arr[k].description)[4*val+3] == '0') 
     return 0; 

		 /* we have to examine previously set members of the array arr2 */

		 for(j=1; j<=i; j++) {

				 pj=((*(p-j)).v)[0]-'0';
				 dj=((*(p-j)).v)[1]-'0';
				 tj=((*(p-j)).v)[2]-'0';
				 vj=(*(p-j)).value -'0'; 
				 
 /* these are the digits from the previous member and the value assigned */

			

	 if((arr[k].description)[4*pj+first] == '1' && (arr[k].description)[4*dj+second] == '1' && 
		 (arr[k].description)[4*tj+third] == '1' && (arr[k].description)[4*vj+val] == '0')
		  {res=0; break;}

	 
    if((arr[k].description)[4*first+pj] == '1' && (arr[k].description)[4*second+dj] == '1' && 
	 (arr[k].description)[4*third+tj] == '1' && (arr[k].description)[4*val+vj] == '0') 
	 {res=0; break;}

	/* if the index of the previous is in between  12 and 23  there are two additional checks */

	if(i-j<24 && i-j>=12 ){ 
		
		if((arr[k].description)[4*pj+first]=='1'&& (arr[k].description)[4*tj+second]=='1'
	&& (arr[k].description)[4*tj+third]=='1' && (arr[k].description)[4*vj+val]=='0') 
	{ res =0; break;}

		if( (arr[k].description)[4*first+pj]=='1' && (arr[k].description)[4*second+tj]=='1'
	&& (arr[k].description)[4*third+tj]=='1'
	&& (arr[k].description)[4*val+vj]=='0') 
	{ res =0; break;}

		

	}

	
		 
	if(i<24 && i>=12 )	{ 
		
		if((arr[k].description)[4*pj+first]=='1'&& (arr[k].description)[4*dj+third]=='1'
	&& (arr[k].description)[4*tj+third]=='1' && (arr[k].description)[4*vj+val]=='0')
	 { res =0; break;}

		if( (arr[k].description)[4*first+pj]=='1' && (arr[k].description)[4*third+dj]=='1'
	&& (arr[k].description)[4*third+tj]=='1'
	&& (arr[k].description)[4*val+vj]=='0') 
	{ res =0; break;}

		

	} 
		 
		 
		 
		 
		 
		 } /*for*/

		 if(rez ==0) return 0; 

		 return 1; /*ok*/

		  }/*newcheck */
               
	  
	  	
		
		
		
		
		
		
		
		
		\end{verbatim}
		
\newpage

\section{Appendix two -- Source codes for 5--element digraphs, both sequential and parallel mode}

\noindent There are two separate source codes presented here. The first one is in sequential mode and it creates an output text file. The second one runs in parallel mode (it was originally executed on $13$ processors in master--slave hierarchy) reading this output file as an input and creating yet another output file. Both can be executed exactly as given here.

\noindent The first code does the following:

\begin{itemize}
\item first we generate all 5--element digraphs
\item then we find all the non--isomorphic ones (since there are so many digraphs, the method used for this resembles the method for identifying all prime numbers known as 'the sieve of Eratosthen' -- namely  isomorphism flags for all digraphs are previously set to '0', so we take the first digraph , generate all its isomorphic copies and raise their  flags to '1' in the array containing all digraphs.  Then we take the next digraph  having '0' on this flag, generate its copies  and set their flags to '1', and so on...When raising the flags of copies to '1' we didn't search for them, the copies that is,  through the array of all digraphs but rather calculate their positions in it -- after  generating each copy of a digraph we turned  it from string format consisting of '0' and '1'  into a binary number with the same digits, and then into a corresponding decimal number which is exactly the index of the copy in this array.
\item in the end we write all the non--isomorphic digraphs into an output file 
\end{itemize} 

\vspace{ 1 cm}

\noindent This way we come to 291968 non--isomorphic 5--element digraphs.
\vspace{ 1 cm}

\noindent The second code is to be executed in parallel mode and it does the following:

\begin{itemize}
\item we read the output file made by the first code and store these digraphs into an array
\item we set the matrix containing all binary idempotent commutative operations on five nodes  (because of idempotence and commutativity presumed each operation is represented by a 10--element row 
\item we set the matrix of subalgebras for all this digraphs which would help us when checking the compatibility of binary operations with digraphs  
\item  we test digraphs for a wnu2--term which now means checking compatibility  of a digraph given with the operations from this matrix -- once a compatible operation is found  we proceed to the next digraph 
\item digraphs not having a compatible binary commutative operation (wnu2 term) are in the end written into a new output file 
\end{itemize}

\vspace{1 cm}

\noindent This way we come to $132509$ non--isomorphic digraphs having a wnu2 polymorphism.

\vspace{1 cm}

\noindent \bf{The source  code identifying non--isomorphic digraphs is the following:}

\vspace{1 cm}

\begin{verbatim}

// five_vertices_one.cpp : Defines the entry point for the console application.
//

#include "stdafx.h"
#include <iostream>
#include <fstream>
using namespace std;
#include "math.h"
#define NUMBER 33554432 /* the number of 5-element digraphs*/

#define NUMBER1 291968  
FILE *fp; 


/* structures*/

char subalgebras[NUMBER][26];  
// to each digraph we associate a row in this matrix for its subalgebras

struct graph1 { char description[26] ; char isomorph; char majority; 
char wnu2; char wnu3;};
struct graph1 array[NUMBER];
// this is the array for all digraphs

struct graph1 array1[NUMBER1];

// this is the array for non-isomorphic ones


//prototypes of functions used

void add_one(char*, char*);
int num_edges(char*); 
void set_isomorph(char*, long int ); 
void f_subalgebras(void); 
int power( int, int); 



void main()
{
	int  n_1, n_2, iso, j, num_wnu3,k; long int i, num_non_iso = 0; ; 

	  long int auxiliary, aux2, br, num2, num3; 

/*initializing the array of digraphs*/

	 
	for(i=0; i<NUMBER; i++){
		array[i].isomorph='0';
		array[i].majority='0';
		array[i].wnu2='0';
		array[i].wnu3= '0';
		for (j=0; j<25; j++)  (array[i].description)[j] ='0';
		(array[i].description)[25] ='\0';}	

	/*setting relations, that is generating all 5-element digraphs*/

    for(i=1; i<NUMBER; i++)
		add_one(array[i-1].description, array[i].description); 


// by this point we have all 5-element digraphs, the presumed order of edges is:
// <0,0>, <0,1>, <0,2>, <0,3>, <0,4>, ...<4,4>
	   
	
// we search for non-isomorphic digraphs now, 
// copies will have their isomorphism flags raised to '1' 
	
	for( i=1; i<NUMBER-1; i++) 
		if (array[i].isomorph == '0')
		set_isomorph(array[i].description, i); 
		
// now we count non_isomorphic digraphs

	  num_non_iso=0; 
	  for( i=0; i<NUMBER; i++) 
		if (array[i].isomorph == '0')  num_non_iso++; 
		
printf (" \n\n non_isomorphic  : %d\n", num_non_iso);


 /*initializing the array for non_isomorphic digraphs*/

	 
	for(i=0; i<NUMBER1; i++){
		array1[i].isomorph='0';
		array1[i].majority='0';
		array1[i].wnu2='0';
		array1[i].wnu3= '0';
		for (j=0; j<25; j++)  (array1[i].description)[j] ='0';
		(array1[i].description)[25] ='\0';}	

// now writing these into a file

	fp=fopen("C:\graphs\graphs.txt", "w"); 
	
	for(i=0; i<NUMBER1; i++)
		fprintf(fp, "%s\n", array1[i].description); 
	fclose(fp);

	printf("\n\n done"); 


	getchar();
	getchar(); 


} /*main*/



// this function generates the next digraph by adding binary one 
// to the string with the address p and  
// setting this new string on the address q


void add_one( char *p, char *q){

	char *p1=p, *q1 =q; 


	while(*p1){
		*q1 = *p1;
		p1++; q1++;
	}
	 *q1 ='\0';
	 q1--;
	 while(*q1 =='1') q1 --;
	 *q1='1';
	 q1++;
	while( *q1 =='1') { *q1 ='0'; q1++;}
} /*add_one*/


// this function counts edges in the digraph on the address p


int num_edges( char *p) 
{
	   int i=0; 
	   while(*p){ if( *p =='1') i++; p++;} 

	   return i;} /* num_edges*/


// this function sets isomorphism flags to '1'
// for all the copies of the digraph given (the one with the address p)

void set_isomorph( char *p,  long int k)  /* k is the index of this digraph*/
{
	   char aux[26]; int   i,j, p0, p1, p2, p3, p4;  
	   long int index;
	  
	  char *q1, *r1; int flag; 

	   char *perm[119];    
	 
// this array of pointers contains all permutations of nodes '0','1','2','3','4'
// except for the identity one 

	   perm[0] = "01243";
	   perm[1] = "01324";
	   perm[2] = "01342";
	   perm[3] = "01423";
	   perm[4] = "01432";
	   perm[5] = "02134";
	   perm[6] = "02143";
	   perm[7] = "02314";
	   perm[8] = "02341";
	   perm[9] = "02413";
	   perm[10] = "02431";
	   perm[11] = "03124";
	   perm[12] = "03142";
	   perm[13] = "03214";
	   perm[14] = "03241";
	   perm[15] = "03412";
	   perm[16] = "03421";
	   perm[17] = "04123";
	   perm[18] = "04132";
	   perm[19] = "04213";
	   perm[20] = "04231";
       perm[21] = "04312";
	   perm[22] = "04321";
	   perm[23] = "10234";
	   perm[24] = "10243";
	   perm[25] = "10324";
	   perm[26] = "10342";
	   perm[27] = "10423";
	   perm[28] = "10432";
	   perm[29] = "12034";
	   perm[30] = "12043";
	   perm[31] = "12304";
	   perm[32] = "12340";
	   perm[33] = "12403";
	   perm[34] = "12430";
	   perm[35] = "13024";
	   perm[36] = "13042";
	   perm[37] = "13204";
	   perm[38] = "13240";
	   perm[39] = "13402";
	   perm[40] = "13420";
	   perm[41] = "14023";
	   perm[42] = "14032";
	   perm[43] = "14203";
	   perm[44] = "14230";
	   perm[45] = "14302";
	   perm[46] = "14320";
	   perm[47] = "20134";
	   perm[48] = "20143";
	   perm[49] = "20314";
	   perm[50] = "20341";
	   perm[51] = "20413";
	   perm[52] = "20431";
	   perm[53] = "21034";
	   perm[54] = "21043";
	   perm[55] = "21304";
	   perm[56] = "21340";
	   perm[57] = "21403";
	   perm[58] = "21430";
	   perm[59] = "23014";
	   perm[60] = "23041";
	   perm[61] = "23104";
	   perm[62] = "23140";
	   perm[63] = "23401";
	   perm[64] = "23410";
	   perm[65] = "24013";
	   perm[66] = "24031";
	   perm[67] = "24103";
	   perm[68] = "24130";
	   perm[69] = "24301";
	   perm[70] = "24310";
	   perm[71] = "30124";
	   perm[72] = "30142";
	   perm[73] = "30214";
	   perm[74] = "30241";
	   perm[75] = "30412";
	   perm[76] = "30421";
	   perm[77] = "31024";
	   perm[78] = "31042";
	   perm[79] = "31204";
	   perm[80] = "31240";
	   perm[81] = "31402";
	   perm[82] = "31420";
	   perm[83] = "32014";
	   perm[84] = "32041";
	   perm[85] = "32104";
	   perm[86] = "32140";
	   perm[87] = "32401";
	   perm[88] = "32410";
	   perm[89] = "34012";
	   perm[90] = "34021";
	   perm[91] = "34102";
	   perm[92] = "34120";
	   perm[93] = "34201";
	   perm[94] = "34210";
	   perm[95] = "40123";
	   perm[96] = "40132";
	   perm[97] = "40213";
	   perm[98] = "40231";
	   perm[99] = "40312";
	   perm[100] = "40321";
	   perm[101] = "41023";
	   perm[102] = "41032";
	   perm[103] = "41203";
	   perm[104] = "41230";
	   perm[105] = "41302";
	   perm[106] = "41320";
	   perm[107] = "42013";
	   perm[108] = "42031";
	   perm[109] = "42103";
	   perm[110] = "42130";
	   perm[111] = "42301";
	   perm[112] = "42310";
	   perm[113] = "43012";
	   perm[114] = "43021";
	   perm[115] = "43102";
	   perm[116] = "43120";
	   perm[117] = "43201";
	   perm[118] = "43210";

	  
	   
	   for( i=0; i<119; i++) { 
		   for(j=0; j < 25; j++) aux[j] ='0'; aux[25] = '\0'; 

// the i-th permutataion of the array with the address p  is now copied in the array aux,
// and then we calculate the index of this copy in the array array1

        p0 = perm[i][0] - '0';  
		p1 = perm[i][1] - '0';
		p2 = perm[i][2] - '0';
		p3 = perm[i][3] - '0';
		p4 = perm[i][4] - '0';




		if( p[0] =='1') aux[p0*5 + p0] ='1'; 
		if( p[1] =='1') aux[p0*5 + p1 ] ='1'; 
		if( p[2] =='1') aux[p0*5 + p2 ] ='1'; 
		if( p[3] =='1') aux[p0*5 + p3] ='1';
		if( p[4] =='1') aux[p0*5 + p4] ='1'; 
		if( p[5] =='1') aux[p1*5 + p0 ] ='1'; 
		if( p[6] =='1') aux[p1*5 + p1 ] ='1'; 
		if( p[7] =='1') aux[p1*5 + p2 ] ='1'; 
		if( p[8] =='1') aux[p1*5 + p3 ] ='1';
		if( p[9] =='1') aux[p1*5 + p4 ] ='1'; 
		if( p[10] =='1') aux[p2*5 + p0 ] ='1'; 
		if( p[11] =='1') aux[p2*5 + p1 ] ='1';
		if( p[12] =='1') aux[p2*5 + p2 ] ='1'; 
		if( p[13] =='1') aux[p2*5 + p3 ] ='1'; 
		if( p[14] =='1') aux[p2*5 + p4 ] ='1'; 
		if( p[15] =='1') aux[p3*5 + p0 ] ='1';
		if( p[16] =='1') aux[p3*5 + p1 ] ='1';
		if( p[17] =='1') aux[p3*5 + p2 ] ='1';
		if( p[18] =='1') aux[p3*5 + p3 ] ='1';
		if( p[19] =='1') aux[p3*5 + p4 ] ='1';
		if( p[20] =='1') aux[p4*5 + p0 ] ='1';
		if( p[21] =='1') aux[p4*5 + p1 ] ='1';
		if( p[22] =='1') aux[p4*5 + p2 ] ='1';
		if( p[23] =='1') aux[p4*5 + p3 ] ='1';
		if( p[24] =='1') aux[p4*5 + p4 ] ='1';

	

		
		index=0; 

		for(j=0; j<25; j++) if(aux[j]=='1') index+= power(2,24-j);

		 if(index > k) array[index].isomorph ='1'; 

       } /* for i*/ 

	  } /*set_isomorph*/

	

int power( int base, int exp){

	long int res =1; int k; 

	if ( exp == 0) return 1; 
	for ( k=1; k<=exp; k++) res=res*base; 

	return res; }



\end{verbatim}

\vspace{2 cm}

\noindent \bf{The source code examining digraphs for a wnu2 term is the following:}

\vspace{2 cm}



\noindent $\mathbf{Acknowledgements:}$

\noindent  I thank to my phd. adviser Petar Markovi\'c for his suggestions during this research.

\end{document}